\documentclass[11pt]{amsart}

\usepackage{amssymb}
\usepackage{amsthm}
\usepackage{amsmath}
\usepackage{graphicx}
\usepackage{amsaddr}
\usepackage{comment}
\usepackage[usenames,dvipsnames]{xcolor}
\usepackage[margin = 1in]{geometry}
\usepackage{enumerate}
\usepackage[toc,page]{appendix}
\usepackage{lineno}
\usepackage{color}
\usepackage{bbm}
\usepackage{hyperref}
\usepackage{mhchem}
\usepackage{siunitx}

\definecolor{mygreen}{rgb}{0.1,0.75,0.2}

\providecommand{\bbs}[1]{\left(#1\right)}

 \newtheorem{thm}{Theorem}[section]
 \newtheorem{cor}[thm]{Corollary}
 \newtheorem{lem}[thm]{Lemma}

 \newtheorem{rem}[thm]{Remark}

 \numberwithin{equation}{section}

\DeclareMathOperator{\KL}{KL}

\providecommand{\bbs}[1]{\left(#1\right)}
\newcommand{\lra}{\longrightarrow}

\newcommand{\wra}{\rightharpoonup}

\newcommand{\la}{\langle}
\newcommand{\ra}{\rangle}
\newcommand{\pt}{\partial}
\newcommand{\eps}{\varepsilon}
\newcommand{\ud}{\,\mathrm{d}}
\newcommand{\8}{\infty}

\newcommand{\bR}{\mathbb{R}}
\newcommand{\bZ}{\mathbb{Z}}
\newcommand{\bE}{\mathbb{E}}

\newcommand{\sP}{\mathcal{P}}
\newcommand{\sL}{\mathcal{L}}

\newcommand{\rhoe}{\rho_t^\eps}
\newcommand{\pie}{\pi_\eps}
\newcommand{\be}{B_\eps^{-1}}

\newcommand{\lbar}{\overline}

\begin{document}

\title[Homogenization of Wasserstein gradient flows]{Homogenization of 
Wasserstein gradient flows}

\author[Y. Gao]{Yuan Gao}
\address{Department of Mathematics, Purdue University, West Lafayette, 47907}
\email{gao662@purdue.edu}

\author[N. K. Yip]{Nung Kwan Yip}
\address{Department of Mathematics, Purdue University, West Lafayette, 47907}
\email{yipn@purdue.edu (corresponding author)}

\begin{abstract}
We prove the convergence of a Wasserstein gradient flow of a free energy in 
  inhomogeneous media. Both the energy and media can depend on the spatial
variable in a fast oscillatory manner. In particular, we show that the 
gradient-flow structure is preserved in the limit which is expressed in terms 
of an effective energy and Wasserstein metric. 
The gradient flow and its limiting behavior are analyzed through an energy 
dissipation inequality (EDI).
The result is consistent with asymptotic analysis in the realm of
homogenization. However, we note that the effective metric is in general
different from that obtained from the Gromov-Hausdorff convergence of metric
spaces. We apply our framework to a linear Fokker-Planck equation but we
believe the approach is robust enough to be applicable in a broader context.
\end{abstract}

\date{\today}

\maketitle

 \section{Introduction}
 
Optimal transport has appeared in many practical and theoretical
applications, cf. 
\cite{rachev1998mass1, rachev1998mass2, villani2003topics, villani2009optimal, peyre2019computational}. 
Precisely, given a cost function 
$c(\cdot,\cdot): \bR^n\times\bR^n\longrightarrow\bR$, 
and two probability measures $\mu,\nu$ on $\bR^n$, 
the problem of 
optimal transport is to find the minimum cost of transporting
$\mu$ to $\nu$. It has the following two classical
formulations: first by Monge \cite{monge1781memoire} in terms of optimal transport
map, and a second formulation using duality by Kantorovich \cite{kantorovich1942translocation} in terms of
optimal coupling measure:
\begin{equation}
\tag{\text{Monge}}
  \inf\left\{
\int c(x,\Phi(x))\ud\mu(x):\,\,\,\Phi: \bR^n\longrightarrow\bR^n,\,\,\,\Phi_\sharp\mu = \nu
\right\},
\end{equation}
and 
\begin{equation}
\tag{\text{Kantorovich}}
 \inf\left\{
\iint   c(x,y) \ud \gamma(x,y); \quad \int \gamma(x,\ud y)  = \mu(x), \,\, \int \gamma(\ud x,y) = \nu(y)
\right\}.
\end{equation}
In the above, $\gamma$ is a probability measure on the product space 
$\bR^n\times\bR^n$.
The equivalence of the above, under appropriate general assumptions, 
has been established in \cite{pratelli2007equality}. 
Typical examples of cost functions include the Euclidean distance square,
$c(x,y)=|x-y|^2$ which is convex and spatially homogeneous in the sense that
$c(x,y)=c(x-y)$. In this case, the infimum value of the above two formulations
 is    the square of Wasserstein-2 distance between $\mu$ and $\nu$, denoted as $W_2^2(\mu,\nu)$. We refer to \cite{villani2003topics, villani2009optimal, book0, santambrogio2015optimal} for examples
of monographs on the theory of optimal transports.

The main purpose of the current paper is to incorporate spatial inhomogeneity
into the above problem, or more precisely, the cost function $c$. 
We then consider gradient flows with respect to the Wasserstein metric induced 
by $c$ and analyze their limiting behavior or 
description when the inhomogeneity converges in appropriate sense. 
We believe these types of questions appear naturally in many 
applications such as 
urban transportations \cite{bernot2008optimal, buttazzo2008optimal}, 
network science \cite{kivela2014multilayer},
spread of epidemics \cite{balcan2009multiscale}, 
optics \cite{rubinstein2017geometrical},
and many others.
Such a consideration 
indeed has a long history in the realm of homogenization 
\cite{bensoussan2011asymptotic, sanchez1980non}.
On a technical level, we aim to explore how the ideas of homogenization can 
be introduced into optimal transport problems. 
Even though in the current paper
we work in a spatially continuous setting, the problem formulation can 
be posed in a discrete, graph or network setting, as seen from the above
mentioned applications. See also the end of this section for some
mathematical work on these attempts.

To be specific, we consider cost functions $c_\eps(\cdot, \cdot)$ that depend 
on the spatial variables in some oscillatory manner. We find that the 
formulation of Benamou-Brenier \cite{benamou2000computational} is well-suited for this purpose. Not only does
it connect  optimal transport to some underlying ``dynamical process'', 
it allows us to incorporate spatial inhomogeneity ``more or less at will''.
More precisely, we focus on the case that $c_\eps(x,y)$ is defined through a 
{\em least action principle},
\begin{equation}
c_\eps(x,y) = \min \left\{ \int_0^1 L_\eps(\dot{z}_t, z_t) \ud t, 
\quad z: [0,1]\longrightarrow\bR^n,\,\,z_0=x, \,\, z_1=y \right\},
\end{equation}
where we envision that $L_\eps$ is convex in the first variable $ v=\dot{z}_t$ 
and oscillatory or periodic in the second variable $z_t$. 
Note that this cost function also defines a metric in an
inhomogeneous media with periodic structure.
If one further assumes that $L$ is a bilinear form in $v$, given by a 
positive definite matrix $B_\eps(x)$,
\begin{equation}\label{L.bilin}
L(v,z) = \la B_\eps(z)v, \, v\ra,
\end{equation}
then $c_\eps(x,y)$ defines a Riemannian metric on $\bR^n$ 
\begin{equation}\label{eps-c}
c_\eps^2(x,y) = \min\left\{ \int_0^1 \la B_\eps(z_t) \dot{z}_t, \dot{z}_t \ra \ud t, \quad z: [0,1]\longrightarrow\bR^n,\,\, z_0 = x, \,\, z_1=y \right\}.
\end{equation}
The above leads to the following $\eps$-Wasserstein distance (square)
between $\mu,\nu\in \sP(\bR^d)$,
\begin{equation}\label{Wc}
W_\eps^2(\mu,\nu):= \inf\left\{
\iint  c_\eps(x,y) \ud \gamma(x,y); \quad \int \gamma(x,\ud y)  = \mu(x), \,\, \int \gamma(\ud x,y) = \nu(y)
\right\}.
\end{equation} 
The description and formulation in this and next sections is applicable for 
general spatially inhomogeneous $B_\eps$ but the focus of this paper is when $B_\eps$ takes the form $\displaystyle B_\eps(x)=B(\frac{x}{\eps})$ -- see Section \ref{main} for precise statements and assumptions.

In order to keep the technicality in this paper manageable, 
we will only consider probability measures having densities with respect to the Lebesgue measure. Henceforth, for simplicity, we will use $\sP_2(\bR^n)$ to denote these 
measures or their densities. The subscript $2$ means these measures have finite second moments. More precise assumptions will be stated in Section \ref{main}. Now let $(\sP_2(\bR^n), W_\eps)$ be the Polish space endowed with the 
$\eps$-Wasserstein metric. The main questions we want to understand are:
\textit{whether gradient-flow structures in $(\sP_2(\bR^n), W_\eps)$ 
are preserved as $\eps \to 0$ and if so,
what the limiting Wasserstein distance $\lbar{W}$ and 
gradient flow are.} We have given positive results for the case of 
linear Fokker-Planck equations in periodic media.

With \eqref{eps-c}, the $\eps$-Wasserstein distance $W_\eps$ can be expressed 
using the following spatially inhomogeneous Benamou-Brenier formulation,
\begin{equation}\label{WBB}
W_\eps^2(\rho_0, \rho_1) := \inf\left\{
\int_0^1\int \rho_t(x)\langle B_\eps(x) v_t(x), v_t(x)\rangle \,dx\,dt,\quad
(\rho_t, v_t)\in V(\rho_0, \rho_1)
\right\}
\end{equation}
where 
\begin{equation}\label{ContEqn}
V(\rho_0, \rho_1) := \Big\{(\rho_t, v_t):\,\,
\frac{\partial \rho_t}{\partial t} + \nabla\cdot(\rho_t v_t) = 0,
\quad
\rho(\cdot, 0) = \rho_0,\quad
\rho(\cdot, 1) = \rho_1
\Big\}.
\end{equation}
The work \cite{bernard2007optimal} -- see its Theorems A and B -- 
in fact shows that the $\inf$ of \eqref{WBB} (and \eqref{Wc}) is achieved by a unique 
interpolation between $\rho_0$ and $\rho_1$, given by a flow map $\frac{\ud}{\ud t}\Phi^\eps_t=v_t(\Phi^\eps_t)$,
\begin{equation}
\rho_t = (\Phi^\eps_t)_{\sharp}\rho_0,
\quad 0\leq t \leq 1.
\end{equation}
Note that for the case $\eps =1, B_\eps = I$, \eqref{WBB} is the celebrated Benamou-Brenier formula \cite{benamou2000computational} for the standard (squared) Wasserstein distance
 \begin{equation}\label{W2}
 W_2^2(\rho_0,\rho_1)=\inf\left\{
\iint  |x-y|^2 \ud \gamma(x,y); \quad \int \gamma(x,\ud y)  = \rho_0(x)\ud x, \,\, \int \gamma(\ud x,y) = \rho_1(y)\ud y
\right\}.   
 \end{equation}
The functional in \eqref{WBB} defines an action functional   on $(\sP_2(\bR^n), W_2)$, 
which allows one to directly use least action principles on $(\sP_2(\bR^n), W_2)$ to compute the $W_2$-distance.
In the seminal paper \cite{otto2001geometry}, \textsc{Otto} went further to 
regard 
$W_2$ as a   Pseudo-Riemannian  distance on $\sP_2(\bR^n)$ with the Riemannian metric 
being the same as the one given by the Benamou-Brenier formula. 
More precisely, 
for any  $s_1, s_2$ on the tangent plane $T_\sP$ at $\rho\in \sP$, the metric tensor on $T_\sP\times T_\sP$ is given by
\begin{equation} 
\big\la s_1, s_2\big\ra_{T_\sP, T_\sP}:= \int  \rho(x)\langle   \nabla \varphi_1(x), \nabla \varphi_2(x)\rangle \,dx, \quad \text{ where } 
s_i = -\nabla \cdot (\rho \nabla \varphi_i), \,\,\,i=1,2.
\end{equation}
(See Section \ref{sec:FP-GF} for an explanation of going from 
$v_t$ in \eqref{WBB} to $\nabla\varphi$ above.)
 
With the above set-up for the Wasserstein distance, we proceed to consider
gradient flows in $(\sP_2(\bR^n), W_\eps)$ of a given energy functional
$E_\eps: \sP_2(\bR^n) \longrightarrow\bR$,
\begin{equation}\label{eGF1st}
\pt_t \rho^\eps_t = - \nabla^{W_\eps} E_\eps (\rho^\eps_t).
\end{equation}
The precise dynamics is uniquely determined by a dissipation functional 
on the tangent plane characterizing the rate of change of the energy
from which the Wasserstein gradient $\nabla^{W_\eps}$ is derived.
In this paper, we consider energy dissipation expressed by the metric
$W_\eps$ (induced by \eqref{WBB}). It turns out $W_\eps$ can be formally interpreted as
a Riemannian metric (see \eqref{epsRe}),
which in particular is given by a bilinear form.
Based on the expression of $\nabla^{W_\eps}$ (see \eqref{nablaW}), $\eps$-Wasserstein gradient flow \eqref{eGF1st} can be explicitly written as
\begin{equation}\label{eFP1st}
\pt_t \rho^\eps_t 
=\nabla \cdot \bbs{\rho^\eps_t B_\eps^{-1} 
\nabla\frac{\delta E_\eps}{\delta \rho}(\rho^\eps_t)}.
\end{equation}
Note that our formulation allows oscillations in 
both the energy $E_\eps$ and media $B_\eps$.

If the total energy is taken as the relative entropy or the
Kullback–Leibler divergence between $\rho$ and another probability
distribution $\pie \in \sP_2(\bR^n)$,
\begin{equation}\label{Eeps}
E_\eps(\rho) 
= \KL(\rho||\pi_\eps)
:= \int_{\bR^n} \rho(x) \log \frac{\rho(x)}{\pi_\eps(x)} \ud x,
\end{equation}
then the above $\eps$-Wasserstein gradient flow \eqref{eFP1st} is the same as
a linear Fokker-Planck equation with oscillatory coefficients.  
The above energy is often called the free energy of the system
and $\pie$ in \eqref{Eeps} 
is a stationary distribution corresponding to an underlying
stochastic process.

Our main result is the evolutionary convergence of the $\eps$-Wasserstein 
gradient flow \eqref{eFP1st} as $\eps \to 0$, to a limit also characterized as 
a gradient flow of an effective total energy $\overline{E}$ with respect to 
an effective Wasserstein distance $\overline{W}$. The distance $\overline{W}$ 
induced by the evolutionary convergence is still a Riemannian metric on 
$\sP_2(\bR^n)$. However, we find that it is in general different from the 
direct Gromov-Hausdorff limit of $W_\eps$.  
Even though our main result is proven for continuous state spaces, the approach we used for proving the convergence of multi-scale gradient flows can also be applied to discrete state spaces, in particular, graphs with inhomogeneous structure.

The main approach we use is to first recast 
the $\eps$-Wasserstein gradient flow \eqref{eFP1st} as a  generalized gradient 
flow  in the following form of an energy dissipation inequality (EDI)
\begin{equation}\label{EDI1st} 
E_\eps(\rhoe) + \int_0^t \left[ \psi_\eps(\rho^\eps_\tau, \pt_\tau \rho^\eps_\tau ) +  \psi^*_\eps\left(\rho_\tau^\eps, -\frac{\delta E_\eps}{\delta \rho}(\rho_\tau^\eps)\right) \right] \ud \tau  \leq  E_\eps(\rho_0^\eps).
\end{equation}
This formulation involves dissipation functionals $\psi_\eps$ and 
$\psi^*_\eps$ on the tangent and the co-tangent plane of $\sP_2(\bR^n)$, 
respectively. 
Inequality \eqref{EDI1st} is in fact equivalent to the strong form of gradient flow \eqref{eGF1st} since the functional $\psi_\eps$ and $\psi^*_\eps$ are convex conjugate of each other; for details, see Section \ref{sec2.2}.
Then the limiting behavior of the dynamics is obtained by considering  
the limit of the functionals in \eqref{EDI1st}.

The framework using the EDI formulation of gradient flows to obtain the evolutionary $\Gamma$-convergence of  gradient flows was first established by \textsc{Sandier and Serfaty}
\cite{sandier2004gamma, serfaty2011gamma}. In this setting,
the key estimates are the lower bounds of the free energy and the 
energy dissipations in terms of the metric velocity and the metric slope.  Many generalizations of the evolutionary convergence for generalized gradient flow systems are     developed by \textsc{Mielke, Peletier} and collaborators; 
see the concept of energy-dissipation-principle (EDP) convergence of gradient flows in \cite{arnrich2012passing, liero2017microscopic}, the concept of generalized tilt/contact EDP convergence developed in \cite{dondl2019gradient, mielke2021exploring}, and also the review \cite{mielke2016evolutionary}.
 
Following the above general framework for evolutionary $\Gamma$-convergence of 
gradient flows, we pass the limit in $\eps$-EDI \eqref{EDI1st} 
by proving the lower bounds of 
all three functionals on the left-hand-side of \eqref{EDI1st}:
the energy functional $E_\eps$, the time integrals of dissipation functionals 
$\psi_\eps$ and $\psi^*_\eps$. 
The lower bounds of the latter two, denoted as $\psi$ and $\psi^*$, are still 
functionals in bilinear form and are convex conjugate of each other and thus determines the limiting Wasserstein gradient flow with an effective Wasserstein distance $\overline{W}$; see the precise definition of these lower bounds in   Theorem \ref{thm1}. The lower bound for $\psi^*_\eps$ is obtained by using a Fisher 
information reformulation in terms of 
$\displaystyle \sqrt{\frac{\rho^\eps}{\pi_\eps}}$
\cite{arnrich2012passing, AGS} and a by now classical 
$\Gamma$-convergence technique for an associated Dirichlet energy. 
On the other hand, the lower bound for $\psi_\eps$ is obtained by a relaxation via the Legendre transformation and an upper bound estimate for $\psi^*_\eps$. 
This requires one to overcome some regularity issues brought by the 
oscillations in the energy functional $E_\eps$ and the solution curve 
$\rho^\eps$. This is achieved via a symmetric reformulation of the Fokker-Planck 
equation in terms of the variable 
$\displaystyle f^\eps : =\frac{\rho^\eps}{\pie}.$

We briefly mention some related references on Wasserstein gradient flow with multi-scale behaviors.  Modeling of Fokker-Planck equation as a gradient flow in Wasserstein 
space was first noted by \textsc{Jordan-Kinderlehrer-Otto} \cite{jordan1998variational}. They also show the convergence of a 
variational backward Euler scheme.
There are many other evolutionary problems that can be formulated using multi-scale Wasserstein gradient flows; see for instance the porous medium equation \cite{otto2001geometry} and more general   aggregation-diffusion equations reviewed in \cite{carrillo2019aggregation}. In \cite{arnrich2012passing}, they use the evolutionary convergence of Wasserstein gradient flow to 
analyze the mean field equation in a zero noise limit for a reversible drift-diffusion process. 
There are also extensions for the zero noise limit from diffusion processes to 
chemical reactions described by time-changed Poisson processes on countable 
states; 
see \cite{maas2020modeling} for the reversible case using  a discrete Wasserstein 
gradient-flow approach and \cite{GL22t} for the irreversible case 
using a nonlinear semigroup approach for Hamilton-Jacobi equations.  
Homogenization of action functionals on the space of probability 
measures has also been studied in \cite{gangbo2012homogenization}.
In addition, convergence of Wasserstein gradient flows has been applied to 
related questions which explore the mean-field limit and large deviation principle of weakly interacting 
particles; cf. \cite{dupuis2012large, budhiraja2012large} and some recent developments in \cite{carrillo2020lambda, delgadino2021diffusive}. 
Furthermore, a similar convergence approach has also been used for generalized gradient
flows and optimal transport  on graphs and their diffusive limits. 
In various discrete settings, we refer to 
\cite{gigli2013gromov} for Gromov-Hausdorff convergence of discrete Wasserstein metrics,  \cite{forkert2022evolutionary} for evolutionary $\Gamma$-convergence of finite volume scheme for linear Fokker-Planck equation, \cite{gladbach2020homogenisation, gladbach2023homogenisation} for the homogenization of Wasserstein distance on periodic graphs,  and the recent works \cite{schlichting2022scharfetter, hraivoronska2023diffusive, hraivoronska2023variational} for diffusive limits of some generalized gradient flows on graph.

 The remainder of this paper is outlined as follows. 
In Section \ref{sec2}, we introduce the inhomogeneous Fokker-Planck and the $\eps$-Wasserstein gradient flow in EDI form and describe   our assumptions and main results. In Section \ref{sec3}, we obtain some uniform regularity estimates and convergence results for the $\eps$-Wasserstein gradient flow. In Section \ref{epsEDI-0EDI}, we  pass the limit in the EDI form of the $\eps$-Wasserstein gradient flow by proving lower bounds for the free energy and two dissipation functionals; see Theorem \ref{thm1}. In Section \ref{sec5}, we study the limiting gradient flow with respect to the induced limiting Wasserstein metric and compare it with the usual Gromov-Hausdorff convergence of $W_\eps.$

\section{$\eps$-system: inhomogeneous Fokker-Planck and generalized gradient flow}\label{sec2}
In this section, we introduce  a spatially inhomogeneous Fokker-Planck equation, which, with fixed $\eps>0$, can be recast as a generalized gradient flow in $\eps$-Wasserstein space in terms of a total energy given by a relative entropy. This Fokker-Planck equation is motivated by a drift-diffusion process with inhomogeneous noise and drift that satisfy the fluctuation-dissipation relation. 
In Section \ref{sec2.2}, we choose a pair of 
quadratic dissipation functionals $(\psi_\eps, \psi^*_\eps)$
which are convex conjugate to each other to recast the
$\eps$-Fokker-Planck equation as a generalized gradient flow in 
an EDI form. Then in Section \ref{main}, we state and explain our main results on the convergence of the gradient-flow structure as $\eps \to 0$ and the resulting homogenized gradient flow of an effective free energy $\overline{E}$ with respect to an effective Wasserstein metric $\overline{W}$.

From now on, to avoid boundary effects, we work on periodic domain, denoted as $\Omega:=\mathbb{T}^n.$   Given any smooth potential function $U_\eps: \Omega\longrightarrow\bR$,
consider the following (free) energy functional on $\sP(\Omega)$
\begin{equation}\label{freeEeps}
E_\eps(\rho) = \int_\Omega U_\eps(x) \rho(x) \ud x + \int_\Omega \rho(x) \log \rho(x) \ud x.
\end{equation}
Let
\begin{equation}\label{pi}
\pi_\eps(x) = e^{-U_\eps(x)}.
\end{equation}
Then \eqref{freeEeps} can be written in the form \eqref{Eeps}.
The first variation
 $\displaystyle \frac{\delta E_\eps}{\delta \rho}$ of $E_\eps$ is then given by,
\begin{equation}
\frac{\delta E_\eps}{\delta \rho}(\rho) = \log \rho + 1 + U_\eps = \log \frac{\rho}{\pi_\eps} +1. 
\end{equation}
With a positive definite matrix $B_\eps$,
we consider  the following inhomogeneous Fokker-Planck equation 
\begin{equation}\label{epsFP}
\pt_t \rho^\eps_t = \nabla \cdot \bbs{ \rho^\eps_t B_\eps^{-1} \nabla \frac{\delta E_\eps}{\delta \rho}(\rho^\eps_t) } = \nabla \cdot \bbs{\be \nabla \rhoe + \rhoe \be \nabla U_\eps}.
\end{equation}
The above equation can be interpreted in two ways. 
One is to regard it as the Kolmogorov forward equation of a 
drift-diffusion process with a multiplicative noise, while another as
a gradient flow in a Wasserstein space $(\sP(\Omega), W_\eps)$ with the cost function defined in \eqref{eps-c}.
We describe both of these in the following.

\subsection{$\eps$-Fokker-Planck equation \eqref{epsFP} as a Kolmogorov
equation}
Consider a drift-diffusion process $(X_t)_{t\geq 0}$, described by  
the following stochastic differential equation 
\begin{equation}\label{sde}
\ud X_t = b(X_t) \ud t + \sigma(X_t) * \ud B_t,
\end{equation}
where $B_t$ is a one-dimensional Brownian motion, and
\begin{equation}
b(x) = - \be(x) \nabla U_\eps(x), \quad \text{and}\quad
\sigma(x) = \sqrt{2 \be(x)}.
\end{equation}
Here  the multiplicative noise $\sigma(X_t) * \ud B_t$ is in the backward Ito 
differential sense, which is equivalent to the forward Ito differential
by adding an additional drift term
$$
\sigma(X_t)* \ud B_t = \frac12\nabla \cdot (\sigma \sigma^T)(X_t) \ud t + 
\sigma(X_t) \ud B_t.
$$
By Ito's formula, the generator of the process $(X_t)_{t\geq 0}$ is 
derived as follows. For any test function $\varphi\in C^2_b(\bR^n)$ and
initial condition $X_0=x$, we compute
\begin{equation}
\begin{aligned}
\lim_{t\to 0^+} \frac{\bE^x[\varphi(X_t)]-\varphi(x)}{t} 
= &\lim_{t\to 0^+} \bE^x \frac{1}{t}\int_0^t \big[ \nabla \varphi(X_s) \cdot b(X_s) \\
&+ \frac12  \nabla^2 \varphi(X_s) : (\sigma\sigma^T)(X_s) + \frac12 (\nabla \cdot (\sigma\sigma^T)(X_s)) \cdot \nabla \varphi (X_s)  \big] \ud s\\
=& \nabla \varphi(x) \cdot b(x) + \frac12 \nabla \cdot (\sigma\sigma^T \nabla \varphi(x))=: \mathcal{L}\varphi.
\end{aligned}
\end{equation}
Thus the corresponding Fokker-Planck equation to \eqref{sde} is given by
\begin{eqnarray}
\pt_t \rhoe &=& \mathcal{L}^*\rho^\eps_t \nonumber\\
&:=& \frac12 \nabla \cdot \bbs{\sigma \sigma^T \nabla \rhoe}  -\nabla \cdot \bbs{\rhoe b}\nonumber\\
& = & \nabla \cdot \bbs{\be(x) \nabla \rhoe(x)} + \nabla\cdot (\rhoe(x) \be(x) \nabla U_\eps(x) ),
\label{Kol.eqn}
\end{eqnarray}
which is exactly \eqref{epsFP}. Note that the $\pie$ defined in \eqref{pi}, 
which is in the form of a Gibbs measure, is in fact the unique stationary 
distribution of \eqref{Kol.eqn}, $\mathcal{L}^*\pie = 0$.

We remark that in the above drift-diffusion process, we used the Ito backward 
differential to ensure that our process $(X_t)_{t\geq 0}$  with a  multiplicative 
noise is reversible so that one can have a gradient flow structure for the corresponding Fokker-Planck equation. More precisely, we have that the
diffusion  process $(X_t)_{t\geq 0}$ \eqref{sde} starting from $X_0 \sim \pi_\eps$ is reversible in the sense that the time reversed process  has the same distribution, i.e.
\begin{equation}
\bE(\varphi_1(X_t)\varphi_2(X_0)|X_0\sim \pie) = \bE(\varphi_1(X_0)\varphi_2(X_t)|X_0\sim \pie), \quad \forall \varphi_1,\varphi_2\in C_0^\8(\bR^n),\, \forall t>0.
\end{equation}
This condition is equivalent to the symmetry of the generator $\sL$ in $L^2(\pi_\eps)$; cf. \cite{GLL_reversible}.

\subsection{$\eps$-Fokker-Planck equation \eqref{epsFP} as a gradient flow in 
$(\sP(\Omega), W_\eps)$}\label{sec:FP-GF}
Following Otto's formal Riemannian calculus on Wasserstein space 
\cite{otto2001geometry}, we now interpret the Fokker-Planck equation as a (negative) gradient flow in $(\sP(\Omega), W_\eps)$. For this purpose, we need to compute the Wasserstein gradient 
$\nabla^{W_\eps} E_\eps$ of $E_\eps$ in $(\sP(\Omega), W_\eps)$.
 
Given any absolutely continuous curve $\tilde{\rho}_t$ in 
$(\sP(\Omega), W_\eps)$ given by $\tilde{\rho}_t := (\chi_t)_\# \rho$
with $\tilde{\rho}_{t=0}=\rho$,
where $\chi_t$ is the flow map induced by a smooth velocity field
$v_t$. 
Then $\tilde{\rho}_t$ satisfies the continuity equation
$$\pt_t \tilde{\rho}_t + \nabla \cdot \bbs{\tilde{\rho}_t v_t  }=0.$$
With this, we compute the first variation of $E_\eps$ 
\begin{equation}\label{otto1}
\frac{\ud}{\ud t}\Big|_{t=0} E_\eps(\tilde{\rho}_t) 
= \int_\Omega \frac{\delta E_\eps}{\delta \rho} \pt_t \tilde{\rho}_t\big|_{t=0} \ud x 
= \int_\Omega \frac{\delta E_\eps}{\delta \rho} \left(
-\nabla\cdot(\tilde{\rho}_tv_t)\big|_{t=0}
\right)\ud x 
= \int_\Omega \left\la  \nabla \frac{\delta E_\eps}{\delta \rho}, v_0 \right\ra \rho \ud x.
\end{equation} 
We will use the above to identify the gradient $\nabla^{W_\eps} E_\eps$ of 
$E_\eps$ with respect to a Riemannian metric 
$\la \cdot, \cdot \ra_{T_{\sP}, T_{\sP}}$ on the tangent plane 
$T_{\sP}$ of $(\sP(\Omega), W_\eps)$. 

Based on \eqref{WBB}, we have that
for any $\rho\in \sP(\Omega)$ and $s_1, s_2\in T_{\sP}$ at $\rho$, the metric
is given by
\begin{equation}\label{epsRe}
\big\la s_1, s_2\big\ra_{T_\sP, T_\sP}:= \int  \rho(x)\left\langle B^{-1}_\eps(x) \nabla \varphi_1(x), \nabla \varphi_2(x)\right\rangle \,dx, \quad \text{ where } 
s_i = -\nabla \cdot (\rho B^{-1}_\eps\nabla \varphi_i),\,\,\,i=1,2.
\end{equation}
A word is in place here to explain going from $v_t$ in 
\eqref{WBB} to $\nabla\varphi$ above. At a fixed $t$ and $\rho_t$, upon minimizing 
$\displaystyle \int_\Omega\rho_t\la B_\eps(x) v_t, v_t\ra\,\ud x$
over $v_t$ subject to 
$\displaystyle 
- \nabla\cdot(\rho_t v_t)
= s \left(:=\frac{\partial \rho_t}{\partial t}\right)
$, we have that
\[
\int_\Omega\rho_t\la B_\eps(x) v_t, \xi\ra\,\ud x = 0
\,\,\,\text{for all smooth vector field $\xi$ satisfying
$- \nabla\cdot(\rho_t \xi) = 0$}.
\]
Hence $B_\eps v_t$ is orthogonal to all divergence free vector field
of the form $\rho_t\xi$. We then conclude that $B_\eps v_t$ must be the
gradient of some (potential) function $\varphi$. Thus $v_t$ can be represented 
as $v_t = B^{-1}_\eps\nabla\varphi$.

With the above, we express the first variation of $E_\eps$ 
using $\nabla^{W_\eps}E_\eps$ as follows, 
\begin{equation}\label{otto2}
\frac{\ud}{\ud t}\Big|_{t=0} E_\eps(\tilde{\rho}_t) 
= \Big\la \nabla^{W_\eps} E_\eps, \pt_t \tilde{\rho}_t\big|_{t=0} 
\Big\ra_{T_{\sP}, T_{\sP}} 
=  \int_\Omega \rho\la B^{-1}_\eps  \nabla \tilde{\varphi} , \nabla \varphi_0 \ra \ud x, 
\end{equation}
where 
\begin{equation}\label{grad.identify}
\pt_t\tilde{\rho}_t\big|_{t=0} = -\nabla \cdot  \bbs{ \rho B^{-1}_\eps\nabla \varphi_0 }
\quad\text{and}\quad
\nabla^{W_\eps} E_\eps (\rho)= -\nabla \cdot\bbs{\rho B^{-1}_\eps\nabla \tilde{\varphi} }.
\end{equation}
Comparing  \eqref{otto1} with \eqref{otto2}, we have 
\[
\int_\Omega \left\la  \nabla \frac{\delta E_\eps}{\delta \rho}, v_0 \right\ra \rho \ud x
=
\int_\Omega \rho\big\la B^{-1}_\eps  \nabla \tilde{\varphi} , \nabla \varphi_0 
\big\ra \ud x
\]
which is set to hold for any $v_0=B^{-1}_\eps\nabla\varphi_0$.
Hence $\displaystyle \nabla \tilde{\varphi} = \nabla \frac{\delta E_\eps}{\delta \rho}$. Thus the second part of \eqref{grad.identify} leads to the
following identification of $\nabla^{W_\eps} E (\rho)$,
\begin{equation}\label{nablaW}
\nabla^{W_\eps} E_\eps (\rho) := -\nabla \cdot \bbs{\rho B_\eps^{-1} \nabla\frac{\delta E_\eps}{\delta \rho}}=-\nabla \cdot \bbs{\rho B_\eps^{-1} \nabla\log\frac{\rho}{\pie}}.
\end{equation}
Hence the inhomogeneous Fokker-Planck equation \eqref{epsFP} indeed can be 
written as a gradient flow of $E_\eps$ with respect to the $\eps$-Wasserstein metric $W_\eps$, i.e.,
\begin{equation}\label{epsGF}
\pt_t \rho^\eps_t = -\nabla^{W_\eps} E_\eps (\rho^\eps_t)
= \nabla \cdot \bbs{\rho^\eps B_\eps^{-1} \nabla\log\frac{\rho^\eps}{\pi_\eps}}.
\end{equation}

We remark that in general an equation may have many different gradient flow 
structures with respect to the same free energy $E_\eps$, cf. 
\cite{mielke2021exploring}. However, in this paper, we restrict ourselves within   the framework of Wasserstein gradient flows as it fits naturally   to the evolution in probability space.

\subsection{$\eps$-generalized gradient flow in energy-dissipation inequality (EDI) form}\label{sec2.2}

As mentioned previously, in order to study the limiting gradient flow 
structure as the small parameter 
$\eps\to 0$ in our $\eps$-gradient flow \eqref{epsGF}, we will recast it in an 
energy-dissipation inequality (EDI) form \eqref{EDI1st} that is shown to be equivalent to the
original  $\eps$-gradient flow system.

Denote the $\eps$-dissipation on the tangent plane $T_{\sP}$ as a functional $\psi_\eps: \sP \times T_{\sP} \to \bR$ defined by
\begin{equation}\label{psi}
\psi_\eps(\rho, s):= \frac12 \int_\Omega \la \nabla u, B_\eps^{-1} \nabla u \ra \rho \ud x, \quad \text{ with }\, s=-\nabla \cdot \bbs{\rho B_\eps^{-1} \nabla u},
\end{equation}
and the $\eps$-dissipation on the cotangent plane $T^*_{\sP}$ as a functional $\psi^*_\eps: \sP \times T^*_{\sP} \to \bR$ defined by
\begin{equation}\label{c_psi}
\psi^*_\eps(\rho, \xi):= \frac12 \int_\Omega \la \nabla \xi, B_\eps^{-1} \nabla \xi \ra \rho \ud x.
\end{equation}
It is easy to check that 
\begin{equation}
\begin{aligned}
\psi_\eps(\rho,s) =& \sup_{\xi\in T^*_\rho} \Big\{ \la \xi, s \ra_{T^*_\rho, T_\rho} - \psi^*_\eps(\rho, \xi)  \Big\}\\
=& \la \xi^*, s \ra_{T^*_\rho, T_\rho} - \psi^*_\eps(\rho, \xi^*) \quad \text{ with } s= -\nabla \cdot \bbs{\rho B_\eps^{-1}\nabla \xi^* }\\
=&  \frac12 \int_\Omega \la \nabla \xi^*, B_\eps^{-1} \nabla \xi^* \ra \rho \ud x.
\end{aligned}
\end{equation}

Applying the Fenchel-Young inequality to the convex functionals 
$\psi_\eps$ and $\psi^*_\eps$, we have
\begin{equation}\label{FenchelYoung}
\la\xi, s\ra \leq \psi_\eps^*(\rho, \xi) + \psi_\eps(\rho, s),
\quad\text{for all}\quad \xi\in T_\rho^*,\,\,\,\text{and}\,\,\,s\in T_\rho,
\end{equation}
with equality holds if and only if
$\xi\in\partial_s\psi_\eps(\rho,s)$ and 
$s\in\partial_\xi\psi^*_\eps(\rho,\xi)$.
Here 
$\partial_s\psi_\eps(\rho, s)$ 
and $\partial_\xi\psi_\eps^*(\rho, \xi)$ refer
to the sub-differentials of $\psi_\eps$ and $\psi_\eps^*$ on
$T_\rho$ and $T_\rho^*$, respectively, at a fixed $\rho$.
We also note the following.
\begin{enumerate}
\item For all $\eta\in T^*_{\sP}$, we have
\begin{eqnarray*}
&&\Big\langle\partial_\xi\psi^*_\eps(\rho, \xi),\eta\Big\rangle
=
\lim_{\tau\to0}\left.
\frac{d}{d\tau}\psi^*_\eps(\rho, \xi+\tau\eta)\right|_{\tau=0}\\
&=& 
\int\langle\nabla\xi, B_\eps^{-1}\nabla\eta\rangle\rho\,\ud x
=\int-\eta\nabla\cdot(\rho B_\eps^{-1}\nabla\xi)\,\ud x
\end{eqnarray*}
so that $\partial_\xi\psi^*_\eps(\rho,\xi)=-\nabla\cdot(\rho B_\eps^{-1}\nabla\xi)$.
Hence $s\in\partial_\xi\psi^*_\eps(\rho, \xi)$ means
$s=-\nabla\cdot(\rho B_\eps^{-1}\nabla\xi)$.
\item
For all $\sigma\in T_{\sP}$, we have
\begin{eqnarray*}
&&\Big\langle\partial_s\psi_\eps(\rho, s),\sigma\Big\rangle
=
\lim_{\tau\to0}\left.
\frac{d}{d\tau}\psi_\eps(\rho, s+\tau\sigma)\right|_{\tau=0}\\
&=& \int\langle\nabla u, B_\eps^{-1}\nabla\omega\rangle\rho\,\ud x
= \int-u\nabla\cdot(\rho B_\eps^{-1}\nabla\omega)\,\ud x\\
&&
\Big(\text{where}\,\,\,
s=-\nabla\cdot(\rho B_\eps^{-1}\nabla u),\,\,\,
\sigma=-\nabla\cdot(\rho B_\eps^{-1}\nabla\omega)
\Big)
\\
&=& \int u\sigma\,\ud x
\end{eqnarray*}
so that $\partial_s\psi_\eps(\rho,s)=u$. 
Hence $\xi\in\partial_s\psi_\eps(\rho, s)$ means
$\xi$ satisfies $s=-\nabla\cdot(\rho B_\eps^{-1}\nabla\xi)$.
\end{enumerate}

With the above, we now reformulate \eqref{epsGF} in the form of an EDI. 
To this end, we compute, 
\begin{equation}\label{eng.diss.eqn}
\frac{d}{dt}E_\eps(\rhoe) 
= \left\langle\frac{\delta E_\eps}{\delta\rho},\pt_t\rhoe\right\rangle,
\,\,\,\text{or}\,\,\,
\frac{d}{dt}E_\eps(\rhoe) 
+ \left\langle-\frac{\delta E_\eps}{\delta\rho},\pt_t\rhoe\right\rangle = 0.
\end{equation}
By \eqref{FenchelYoung}, $\pt_t\rhoe = -\nabla^{W_\eps}E_\eps(\rhoe)=\nabla \cdot \bbs{\rho B_\eps^{-1} \nabla\frac{\delta E_\eps}{\delta \rho}}$ 
if and only if 
\[
\psi_\eps(\rho^\eps_\tau, \pt_\tau \rho^\eps_\tau ) +  \psi^*_\eps\left(\rho_\tau^\eps, -\frac{\delta E_\eps}{\delta \rho}(\rho_\tau^\eps)\right)
\leq \left\langle-\frac{\delta E_\eps}{\delta\rho},\pt_t\rhoe\right\rangle.
\]
Hence, upon integrating \eqref{eng.diss.eqn}, our gradient flow \eqref{epsGF} 
is equivalent to the following:
\begin{equation}\label{epsEDI}
E_\eps(\rhoe) + \int_0^t \left[ \psi_\eps(\rho^\eps_\tau, \pt_\tau \rho^\eps_\tau ) +  \psi^*_\eps\left(\rho_\tau^\eps, -\frac{\delta E_\eps}{\delta \rho}(\rho_\tau^\eps)\right) \right] \ud \tau  \leq  E_\eps(\rho_0^\eps). 
\end{equation}
We note that the very first step, \eqref{eng.diss.eqn} is a crucial chain rule of differentiation. This is justified in our paper due to the regularity property of our energy functional and the solution.
Precise statements will be given in Section \ref{sec3}. In general (for example, discrete or general metric space) settings, the absolute continuity of $E_\eps(\rhoe)$ (in time) and the validity of the chain rule  \eqref{eng.diss.eqn} need to be proved; cf., \cite{Tse1, hraivoronska2023variational}.

Before leaving this section, we write down the following explicit expressions.
\begin{eqnarray}
\psi^*_\eps\left(\rho_\tau^\eps, -\frac{\delta E_\eps}{\delta \rho}(\rho_\tau^\eps)\right) 
&=& \int_\Omega \left\langle
\nabla\left(\frac{\delta E_\eps}{\delta\rho}\right),
B_\eps^{-1}
\nabla\left(\frac{\delta E_\eps}{\delta\rho}\right)
\right\rangle\rhoe
\,\ud x\nonumber\\
&=&
\frac12 \int_\Omega \left\la \nabla  \log \frac{\rho^\eps_\tau}{\pie}, \be \nabla \log \frac{\rho^\eps_\tau}{\pie} \right\ra \rho^\eps_\tau \ud x,
\end{eqnarray}
and
\begin{equation}
\psi_\eps\left(\rho^\eps_\tau, \pt_\tau\rho^\eps_\tau\right)
= \frac12 \int_\Omega \la \nabla u, B_\eps^{-1} \nabla u \ra \rho_\tau^\eps 
\ud x, \quad \text{ with }\, -\nabla \cdot \bbs{\rho^\eps_\tau B_\eps^{-1} \nabla u} = \pt_\tau\rho_\tau^\eps.
\end{equation}

\subsection{Main results}\label{main}
Briefly stated, our main result is that the gradient-flow structure is
preserved in the limit, i.e., \eqref{epsGF} converges to a limiting
gradient flow. More precisely, the solution $\rho^\eps_t$ of 
\eqref{epsGF} converges (weakly) to $\rho_t$ that solves a gradient flow
with respect to a limiting Wasserstein distance $\lbar{W}$,
\begin{equation}\label{0GF}
\pt_t \rho_t 
= -\nabla^{\lbar{W}}\lbar{E}(\rho_t) 
= \nabla \cdot \bbs{\rho_t \lbar{B}^{-1} \nabla\log\frac{\rho_t}{\lbar{\pi}}}.
\end{equation}
In the above, the limiting energy is given as
\begin{equation}\label{0Eng}
\lbar{E}(\rho) = \text{KL}(\rho||\lbar{\pi}) = 
\int_\Omega \rho\log\frac{\rho}{\lbar{\pi}}\,\ud x,
\end{equation}
where the $\lbar{\pi}$ is simply the spatial average of $\pi_\eps$
with respect to some fast variable -- see \eqref{pi.ave} below.
The matrix $\lbar{B}$ is obtained by taking
appropriate average of $B_\eps$ over the fast variable weighted by the 
solution of a cell problem \eqref{homog.B} or equivalently, by considering 
the $\Gamma$-limit of a variational functional (Theorem \ref{GammaMainThm}).
The Wasserstein distance $\lbar{W}$ is related to $\lbar{B}$ just as
the way $W_\eps$ is related to $B_\eps$ -- see Section \ref{WepsWbar}. 

Similar to \eqref{epsEDI}, \eqref{psi}, and \eqref{c_psi}, 
equation \eqref{0GF} is formulated as an EDI, i.e.,
\begin{equation}\label{0EDI}
\lbar{E}(\rho_t) + \int_0^t \left[ \psi(\rho_\tau, \pt_\tau \rho_\tau ) 
+  \psi^*\left(\rho_\tau, -\frac{\delta \lbar{E}}{\delta \rho}(\rho_\tau)\right) \right] 
\ud \tau  \leq  \lbar{E}(\rho_0),
\end{equation}
where $\psi^*: \sP \times T^*_{\sP} \to \bR$ is the limiting dissipation 
functional on the cotangent plane $T^*_{\sP}$ given by
\begin{equation}\label{c_psiL0}
\psi^*(\rho, \xi):= \frac12 \int_\Omega \la \nabla \xi, \bar{B}^{-1} \nabla \xi \ra \rho \ud x,
\end{equation}
and $\psi: \sP \times T_{\sP} \to \bR$ is the limiting dissipation 
functional on the tangent plane $T_{\sP}$ given by
\begin{equation}\label{psiL0}
\psi(\rho, s):= \frac12 \int_\Omega \la \nabla u, \bar{B}^{-1} \nabla u \ra \rho \ud x, \quad \text{ with }\, s=-\nabla \cdot \bbs{\rho \bar{B}^{-1} \nabla u}.
\end{equation} 
The precise statement of the convergence of \eqref{epsEDI} to \eqref{0EDI}
will be given in Section \ref{epsEDI-0EDI}, Theorem \ref{thm1}.

Curiously, under the current setting, $\lbar{W}$ is {\em not} the 
Gromov-Hausdorff limit $W_{\text{GH}}$ of $W_\eps$ which is the common
mode of convergence for metric spaces,  
cf. \cite{villani2009optimal, gigli2013gromov, GKM20}. 
In Section \ref{sec:GH},
We have constructed examples such that $\lbar{W}$ is {\em strictly bigger} 
than $W_{\text{GH}}$. We believe that this statement is true for general
heterogeneous media.

Before proceeding further, we introduce the following notations and 
conventions. As we will often consider functions that oscillate on a 
small length scale, $0 < \eps \ll 1$, it is convenient to introduce the 
following fast variable
\begin{equation}\label{fast.var}
y=\frac{x}{\eps}.
\end{equation}
The domain for $y$ is taken to be the $n$-dimensional torus $\mathbb T^n$
when the oscillatory functions are $1$-periodic in $y$.
The notation $\lbar{A}$ means that it is derived from some averaging of $A$ 
over the fast variable $y$. For time dependent problems, we often deal with
functions defined on both space and time variables $x,t$. For ease of notation, 
given a function $f=f(x,t)$, we often use $f_t$ to denote $f_t(\cdot)$, i.e., 
the slice of $f$ at a fixed time $t$.
We will use $\wra$ and $\lra$ to denote weak and strong 
convergence in some function spaces. 
Two common spaces used are the space of probability 
measures $\sP(\Omega)$ and $L^p(\Omega)$ spaces. The value of $p$ will depend on contexts. 
For the convergence of a sequence of functions $f_\eps$ as $\eps\to0$, 
we will use the same notation even if the convergence only holds upon
extraction of subsequence. (The convergence can be established for the whole
sequence if the limiting equation has unique solution which is the case
for our linear Fokker-Planck equation \eqref{0GF}.)

Next we state the main assumptions for our results.
Some of these are made only for simplicity. They can be  
relaxed if we choose to use more technical tools.
\begin{enumerate}[(i)]
\item
Recall that the domain $\Omega$ is taken to be an $n$-dimensional
torus ${\mathbb T}^n$.
This is not to be confused with the ${\mathbb T}^n$ for the fast variable $y$. We note that the boundedness of the domain can be removed, allowing one to work in $P_2(\bR^n)$ if a confinement potential $U$ is incorporated in the dynamics. Other boundary conditions, such as Dirichlet or no-flux conditions, may also be considered.

\item
For $B_\eps$, we consider
\begin{equation}
\label{Bform}
B_\eps(x) = B\left(\frac{x}{\eps}\right),\,\,\,\text{or}\,\,\,
B_\eps(x) = B(y),
\end{equation}
where $B(\cdot)$ is $1$-periodic.
Furthermore, $B(\cdot)$ is bounded and uniformly positive definite, i.e.,
there are $C_1, C_2 > 0$ such that for all $y\in\mathbb T^n$, it holds that
\begin{equation}
C_1 I \leq B(y) \leq C_2 I.
\end{equation}
This form of $B_\eps$ can certainly be generalized to allow for dependence on the slow variable: $B_\eps(x)=B(x,\frac{x}\eps)$.
For simplicity, we assume further that $B$ is smooth in $y$.

\item
For $\pie$, we consider the following form of separation of length scales:
\begin{equation}
\label{piform} 
\pie(x) = \pi\left(x,\frac{x}{\eps}\right).
\end{equation}
In the above, $\pi$ is $1$-periodic in the fast variable 
$\displaystyle y=\frac{x}{\eps}$.
We further assume that $\pi$ is smooth in both $x$ and $y$ and 
is bounded away from zero and from above uniformly in $\eps>0$. 
The following notation referring to an averaged version 
of $\pi$ will be used in this paper:
\begin{equation}\label{pi.ave}
\lbar{\pi}(x) = \int \pi(x,y)\ud y.
\end{equation}
As concrete examples, $\pie$ can be taken as
\begin{equation}\label{pie}
\pie^{\text{I}}(x) = \pi_0(x) + \pi_1\left(x,\frac{x}{\eps}\right),
\quad\text{or}\quad
\pie^{\text{II}}(x) = \pi_0(x) + \eps\pi_1\left(x,\frac{x}{\eps}\right).
\end{equation}
Then  $\pie^{\text{I}}$ and $\pie^{\text{II}}$ converge as follow:
\begin{equation}
\pie^{\text{I}}(x) \wra \lbar{\pi}^{\text{I}}(x) :=
\pi_0(x) + \int_{\mathbb T^n}\pi_1(x,y)\,\ud y
,\,\,\,\text{and}\,\,\,
\pie^{\text{II}}(x) \lra \lbar{\pi}^{\text{II}}(x) 
:= \pi_0(x).
\end{equation}
We thus call $\pie^{\text{I}}$ the oscillatory case while 
$\pie^{\text{II}}$ the uniform case.
(We refer to the work \cite{dupuis2012large} for large deviations for multiscale diffusion  with $\pie^{\text{II}}$.)

\item	The initial data $\rho^\eps_0$ is bounded away
from zero and from above uniformly in $\eps>0$. It is assumed to be
well-prepared in the following sense,
\begin{equation}\label{id_con}
\text{there is a $\rho_0$ such that as
$\eps\to0$, it holds $\rho^\eps_0 \wra \rho_0$
and
$E_\eps(\rho^\eps_0) \to \lbar{E}(\rho_0), \quad \text{ as } \, \eps \to 0,$}
\end{equation}
where $E_\eps$ and $\lbar{E}$ are given by \eqref{freeEeps} and \eqref{0Eng}.
More precise smoothness requirements on $\rho_0$ will be listed in 
Lemmas \ref{lem_reg}, 
\ref{lem_timeder_reg}, and Corollaries \ref{cor.f} and \ref{cor.f.time.der}.
\end{enumerate}

We have the following remarks about our results.
\begin{rem}\label{rem_thm}
\hspace{1in}
\begin{enumerate}
\item	
As $\pie^{\text{II}}$ can be treated as a special case of $\pie^{\text{I}}$, 
or more generally, of $\pie$, we will concentrate on the proof for $\pie$.
Our result is also consistent with the statement obtained by 
using the asymptotic expansion described in Appendix \ref{asym.exp}.
At the end of that section, we also make
some remarks about the revised statement for $\pie^{\text{II}}$.

\item	The approach we take resembles the work of Forkert-Maas-Portinale 
\cite{forkert2022evolutionary} on the convergence of a finite volume scheme
for a Fokker-Planck equation. By and large, the framework of their (numerical) approximation
enjoys stronger regularity, while our current problem concentrates on
the oscillation of the underlying medium.
\end{enumerate}
\end{rem}

\section{Some a-priori estimates}\label{sec3}

In order to study the asymptotic behavior as $\eps \to 0$, we first establish   
some a-priori estimates for our $\eps$-gradient flow system
\eqref{epsFP} (or \eqref{epsGF}). These would then 
give us the space-time compactness and convergence. These variational estimates 
for linear parabolic equations are standard but we give a brief proof for completeness.

First, we recast \eqref{epsFP} as
\begin{equation}
\pt_t \rhoe = \nabla \cdot \bbs{\pi_\eps B_\eps^{-1} \nabla \frac{\rhoe}{\pi_\eps}}.
\end{equation}
Denote $\displaystyle f^\eps_t := \frac{\rho^\eps_t}{\pi_\eps}$. 
Then $f^\eps_t$ satisfies the following {\em backward} equation
\begin{equation}\label{backward}
\pt_t f^\eps_t = \frac{1}{\pi_\eps} \nabla \cdot \bbs{ \pi_\eps B_\eps^{-1} \nabla f^{\eps}_t } =: L_\eps (f^\eps _t).
\end{equation}
It is easy to verify that $L_\eps$ is self-adjoint in $L^2(\pi_\eps)$, i.e.,
\begin{equation}
\la L_\eps u, v \ra_{\pi_\eps} = \la u, L_\eps v \ra_{\pi_\eps}, \quad \forall u, v \in L^2(\pi_\eps),
\end{equation}
where $\la \cdot, \cdot \ra_{\pie}$ 
denotes the $\pi_\eps$-weighted $L^2$ inner product,
$\displaystyle \la u, v\ra_{\pie} 
:= \int_\Omega u(x)v(x)\pie(x)\ud x$.

We recall here the standing assumptions of
uniform positive definiteness of $B_\eps$ and 
uniform positivity and boundedness of $\pie$ as stated in 
\eqref{Bform} and \eqref{piform} in Section \ref{main}.
We then have the following uniform estimates for $f^\eps_t$.
\begin{lem}\label{lem_reg}
Let $f^\eps_0$ be the initial data for \eqref{backward}. We define,
\begin{eqnarray}
A_0 & := & \sup_{\eps>0}\int_\Omega (f_0^\eps)^2\pi_\eps\ud x,\label{A0}\\
B_0 & := & \sup_{\eps>0}\int_\Omega 
\langle \nabla f_0^\eps, B_\eps^{-1}\pi_\eps\nabla f_0^\eps\rangle\ud x.
\label{B0}
\end{eqnarray}
Let $0 < T < \infty$ be given. We have the following statements.
\begin{enumerate}
\item If $0 < m_0 < \inf f_0^\eps < M_0 < \infty$ on $\Omega$ for 
some finite positive constants $m_0$ and $M_0$, then
$m_0 < \inf f_t^\eps < M_0$ for all $t > 0$.

\item If $A_0 < \infty$, then 
$f^\eps \in L^\infty((0,T);L^2(\Omega))\bigcap L^2((0,T);H^1(\Omega))$ with the following uniform-in-$\eps$ bound: 
for all $0 < t < T$,  
\begin{equation}\label{L2SpaceEst}
\frac12||f_t^\eps||_{\pi_\eps}^2 + \int_0^t \int_\Omega
\langle \nabla f_s^\eps, B_\eps^{-1}\pi_\eps\nabla f_s^\eps\rangle\ud x\ud s
=\frac12||f_0^\eps||_{\pi_\eps}^2 \leq A_0.
\end{equation}

\item If $B_0 < \infty$ (which by Poincare inequality implies 
$A_0 < \infty$), then
$$
f^\eps \in L^\infty((0,T);H^1(\Omega))\bigcap H^1((0,T);L^2(\Omega)) 
$$
with the following uniform-in-$\eps$ bound:
for all $0 < t < T$,  
\begin{equation}\label{H1SpaceTimeEst}
\frac12\int_\Omega
\langle \nabla f_0^\eps, B_\eps^{-1}\pi_\eps\nabla f_0^\eps\rangle\ud x
+ \int_0^t \int_\Omega(\partial_s f_s^\eps)^2\pi_\eps\ud x\ud s
=\frac12\int_\Omega
\langle \nabla f_0^\eps, B_\eps^{-1}\pi_\eps\nabla f_0^\eps\rangle\ud x \leq \frac{B_0}{2}.
\end{equation}
From \eqref{backward} and 
$\displaystyle \int_0^t \int_\Omega(\partial_s f_s^\eps)^2\pi_\eps\ud x\ud s\leq \frac{B_0}{2}$, we also have
\begin{equation}\label{H2SpaceEst}
\sup_{\eps > 0}\int_0^T\int_\Omega
\Big(\nabla\cdot\big(B_\eps^{-1}\pi_\eps\nabla f_s^\eps\big)\Big)^2
\ud x
\ud s < \infty.
\end{equation}
\end{enumerate}
\end{lem}

\begin{proof} Note that
\[
\pt_t f_t^\eps = B_\eps^{-1}:D^2f_t^\eps + 
\frac{1}{\pie}\big\langle\nabla(B_\eps^{-1}\pie),\nabla f_t^\eps\big\rangle.
\]
By the positive definitenss of $B_\eps$, 
statement (1) then follows directly from maximum principle.

Next, both \eqref{L2SpaceEst} and \eqref{H1SpaceTimeEst} follows
from simple energy identity. For the former, we compute
\begin{eqnarray*}
\frac{d}{dt}\frac12||f_t^\eps||^2_{\pi_\eps}
= \int_\Omega f_t^\eps\partial_t f_t^\eps \pi_\eps\ud x
= -\int_\Omega \langle \nabla f_t^\eps, B_\eps^{-1}\pi_\eps \nabla f_t^\eps
\rangle \ud x.
\end{eqnarray*}
Integration in time from $0$ to $t$ gives \eqref{L2SpaceEst}.

For \eqref{H1SpaceTimeEst}, we compute
\begin{eqnarray*}
&&\frac{d}{dt}\frac12\int_\Omega
\langle \nabla f_t^\eps, B_\eps^{-1}\pi_\eps \nabla f_t^\eps\rangle
\ud x
= \int_\Omega
\langle \nabla \partial_tf_t^\eps, B_\eps^{-1}\pi_\eps \nabla f_t^\eps\rangle
\ud x\\
&=& -\int_\Omega
\partial_tf_t^\eps\nabla\cdot\big(B_\eps^{-1}\pi_\eps \nabla f_t^\eps\big)
\ud x
= -\int_\Omega (\partial_tf_t^\eps)^2\pi_\eps \ud x.
\end{eqnarray*}
Integration in time from $0$ to $t$ again gives the result.
Estimate \eqref{H2SpaceEst} follows from definition.
\end{proof}

The above and Fubini's Theorem immediately leads to the following 
compactness results.
\begin{cor} \label{cor.f}
If $B_0 < \infty$, then there is a subsequence $f^\eps$ 
and an $f\in L^2(0,T; L^2(\Omega))$ such that
$f^\eps\longrightarrow f$ in $L^2(0,T; L^2(\Omega))$, i.e.,
\begin{equation}\label{L2SpaceTimeConv}
\int_0^T\int_\Omega|f^\eps_t-f_t|^2\ud x\ud t\rightarrow0.
\end{equation}
Furthermore, we have
\begin{equation}\label{ae.conv}
\int_\Omega|f^\eps_t-f_t|^2\ud x\rightarrow0
\quad \text{for a.e. $t\in[0,T]$.}
\end{equation}
\end{cor}

For our application, we will also need some regularity
estimates for the time derivative of $f^\eps$. Define
$h_t^\eps := \partial_t f^\eps_t$. Then it satisfies the same 
equation \eqref{backward}, i.e.,
\begin{equation}\label{TimeDerEqn}
\pt_t h^\eps_t = \frac{1}{\pi_\eps} \nabla \cdot \bbs{ \pi_\eps B_\eps^{-1} \nabla h^{\eps}_t } =: L_\eps (h^\eps _t).
\end{equation}
As a direct application of Lemma \ref{lem_reg} and Corollary \ref{cor.f}, 
we have the following lemma and corollary.

\begin{lem}\label{lem_timeder_reg}
Let $h^\eps_0 = \partial_t f_t^\eps|_{t=0}$ be the initial data for \eqref{TimeDerEqn}. We define,
\begin{eqnarray}
C_0 & := & \sup_{\eps>0}\int_\Omega (h_0^\eps)^2\pi_\eps\ud x
\,\,\left(=
\sup_{\eps>0}\int_\Omega (\pt_t f_0^\eps)^2\pi_\eps\ud x
\right),\label{C0}
\\
D_0 & := & \sup_{\eps>0}\int_\Omega
\langle \nabla h_0^\eps, B_\eps^{-1}\pi_\eps\nabla h_0^\eps\rangle\ud x
\,\,\left(=
\sup_{\eps>0}\int_\Omega
\langle \nabla (\pt_t f_0^\eps), B_\eps^{-1}\pi_\eps\nabla (\pt_t f_0^\eps)\rangle\ud x\right).
\label{D0}
\end{eqnarray}
Let $0 < T < \infty$ be given. We have the following statements.
\begin{enumerate}
\item If $C_0 < \infty$, then
$h^\eps \in L^\infty((0,T);L^2(\Omega))\bigcap L^2((0,T);H^1(\Omega))$.
In particular, for all $0 < t < T$, we have the following identity,
\begin{equation}\label{L2SpaceEst_timeder}
\frac12||h_t^\eps||_{\pi_\eps}^2 + \int_0^t\int_\Omega
\langle \nabla h_s^\eps, B_\eps^{-1}\pi_\eps\nabla h_s^\eps\rangle\ud x\ud s
=\frac12||h_0^\eps||_{\pi_\eps}^2.
\end{equation}

\item If $D_0 < \infty$, then
$h^\eps \in L^\infty((0,T);H^1(\Omega))\bigcap H^1((0,T);L^2(\Omega))$.
In particular, for all $0 < t < T$, we have the following identity,
\begin{equation}\label{H1SpaceTimeEst_timeder}
\frac12\int_\Omega
\langle \nabla h_t^\eps, B_\eps^{-1}\pi_\eps\nabla h_t^\eps\rangle\ud x
+ \int_0^t \int_\Omega(\partial_s h_s^\eps)^2\pi_\eps\ud x\ud s
=\frac12\int_\Omega
\langle \nabla h_0^\eps, B_\eps^{-1}\pi_\eps\nabla h_0^\eps\rangle\ud x
\end{equation}

\end{enumerate}
\end{lem}

\begin{cor}\label{cor.f.time.der} If $D_0 < \infty$, then there is a subsequence $h^\eps$
and an $h\in L^2(0,T;L^2(\Omega)$ such that
$h^\eps\longrightarrow h$ in $L^2(0,T;L^2(\Omega))$, i.e.
\begin{equation}\label{L2SpaceTimeDerConv}
\int_0^T\int_\Omega|h^\eps_t-h_t|^2\ud x\ud t\rightarrow0.
\end{equation}
Furthermore, we have
\begin{equation}\label{ae.time.der.conv}
\int_\Omega|h^\eps_t-h_t|^2\ud x\rightarrow0,
\,\,\,\text{for a.e. $t\in[0,T]$.}
\end{equation}
\end{cor}

Recall Assumption (iii) in Section \ref{main} for the invariant measure $\pie$.
For the convenience of our upcoming proof, we collect the
necessary convergence results in the following lemma.
\begin{lem}\label{init.ass}
Suppose $A_0, B_0, C_0$ and $D_0 < \infty$.
Then (from Lemmas \ref{lem_reg} and \ref{lem_timeder_reg}) we have
\begin{equation}
f^\eps \in L^\infty((0,T);H^1(\Omega))\bigcap H^1((0,T);L^2(\Omega)),
\quad\text{and}\quad
\partial_t f^\eps  \in L^\infty((0,T);H^1(\Omega))\bigcap H^1((0,T);L^2(\Omega)).
\end{equation}
Furthermore (from Corollaries \ref{cor.f} and \ref{cor.f.time.der}),
up to $\eps$-subsequence, we have
\begin{equation}\label{conv1}
f^\eps  \longrightarrow f ,\quad\text{and}\quad
\pt_tf^\eps  \longrightarrow \pt_tf \quad\text{in $L^2((0,T);L^2(\Omega))$}.
\end{equation}
Upon defining $\rho_t = f_t\lbar{\pi}$, we have
\begin{eqnarray}\label{f.st.L2L2}
\frac{\rho^\eps}{\pie}\,\,\,(=f^\eps) 
&\longrightarrow & \frac{\rho}{\lbar{\pi}}\,\,\,(=f)
\quad\text{in $L^2((0,T);L^2(\Omega))$,}\\
\rho^\eps 
 &\rightharpoonup & 
\rho
\quad\text{in $L^2((0,T);L^2(\Omega))$,}
\end{eqnarray}
and
\begin{eqnarray}\label{ft.wk.L2L2}
\frac{\pt_t\rho^\eps}{\pie}\,\,\,(=\pt_tf^\eps) 
&\longrightarrow & \frac{\pt_t\rho}{\lbar{\pi}}\,\,\,(=\pt_tf)
\quad\text{in $L^2((0,T);L^2(\Omega))$,}\\
\pt_t\rho^\eps 
 &\rightharpoonup & 
\pt_t\rho
\quad\text{in $L^2((0,T);L^2(\Omega))$.}
\label{conv5}
\end{eqnarray}
Instead of strong and weak convergence in 
$L^2(0,T;L^2(\Omega))$,
by \eqref{ae.conv} and \eqref{ae.time.der.conv}, 
statements \eqref{conv1}--\eqref{conv5} also hold 
with the same respective strong and weak topologies in $L^2(\Omega)$
for a.e. $t\in[0,T]$.
\end{lem}

\begin{rem}
Note that currently our approach does require a high
degree of regularity for the initial data. Its existence and 
construction would require the characterization of precise oscillations of the 
solution which in principle can be done by considering second and higher order
cell problems.
However, we believe this requirement can be much relaxed by means of parabolic 
regularity. For example, if $A_0 < \infty$, then $f_t^\eps\in H^1(\Omega)$ for 
some $t>0$ and if $B_0 < \infty$, then  $\partial_t f^\eps_t\in L^2(\Omega)$
for some $t>0$. This can be iterated due to the variational structure of 
equation \eqref{backward}.
Alternatively, we can   opt to utilize some technical results similar to
\cite[p.14, steps (a-c)]{jordan1998variational} and
\cite[Proposition 4.4]{forkert2022evolutionary} in which the initial data
even belongs to $L^1(\Omega)$. For simplicity, in this paper, 
we do not pursue this route, as we consider it beyond the scope of 
homogenization which is our key motivation.
\end{rem}

The final statement in this section gives the time continuity of $\rhoe$
in the standard Wasserstein space $(\sP(\Omega), W_2)$ \eqref{W2}.
\begin{lem}
Assume $E_\eps(\rho_0^\eps)<+\8$. For any $T>0$, let $\rhoe, t\in[0,T]$ be a solution to the $\eps$-gradient flow system \eqref{epsEDI}. 
Then there is $0 < C < \infty$ such that
\begin{equation}\label{equi-w2}
W_2^2(\rho^\eps_t, \rho^\eps_s) \leq C|t-s|, \quad \forall\,  0\leq s\leq t\leq T,
\end{equation}
where $W_2(\cdot, \cdot)$ is the standard $W_2$-distance.
Consequently, there exist  a subsequence 
$\rho^\eps$ and $\rho\in C([0,T]; \sP(\Omega))$ such that
\begin{equation}\label{convergence_W2}
W_2^2(\rhoe, \rho_t) \to 0, \quad \text{ uniformly in } t\in[0,T].
\end{equation} 
\end{lem}
\begin{proof}
First, since $\rhoe, t\in[0,T]$ satisfies \eqref{epsEDI} and $E_\eps(\rho_0^\eps)<+\8$, we have for any  $0\leq s\leq t\leq T$, 
\begin{equation}
\int_s^t \psi_\eps( \rho^\eps_\tau, \pt_\tau \rho^\eps_\tau) \ud \tau <+\8.
\end{equation}
This means for the curve $\rhoe, t\in[0,T]$ with $\pt_t \rho^\eps_t = -\nabla \cdot \bbs{\rhoe \be \nabla u^\eps_t}$, we have
\begin{equation}
\int_s^t \int_\Omega\frac12 \la \nabla u^\eps_\tau, \be \nabla u^\eps_\tau  \ra \rho^\eps_\tau \ud x \ud \tau <+\8.
\end{equation}
For this curve, the velocity in the continuity equation is given by 
$v^\eps_t = \be \nabla u^\eps_t$. From \cite[Theorem 17.2]{book0}, we have
\begin{equation}
\begin{aligned}
W_2^2(\rho^\eps_t, \rho^\eps_s) \leq |t-s| \int_s^t \int_\Omega |v^\eps_\tau|^2 \rho^\eps_\tau \ud x \ud \tau =& |t-s| \int_s^t \int_\Omega |\be \nabla u^\eps_\tau|^2 \rho^\eps_\tau \ud x \ud \tau\\
\lesssim & |t-s|\int_s^t \int_\Omega \la  \nabla u^\eps_\tau, \be  \nabla u^\eps_\tau \ra \rho^\eps_\tau \ud x \ud \tau.
\end{aligned}
\end{equation}
This gives the equi-continuity of $\rhoe$ in $(\sP(\Omega), W_2)$. 

Second, for any $t$ fixed, as $\int_\Omega \rhoe \ud x =1$ and $\Omega$ is 
compact, by \cite[Theorem 8.8]{book0}, the weak* convergence of $\rhoe\in \sP$ 
to some $\rho_t\in \sP$ implies that 
\begin{equation}
W_2(\rho^\eps_t, \rho_t) \to 0.
\end{equation}
We then complete the proof by applying the Arzel\'a-Ascoli Theorem in 
$(\sP(\Omega), W_2)$. 
\end{proof}

\section{Passing limit in EDI formulation of $\eps$-gradient flow}
\label{epsEDI-0EDI}
In this section, we prove that the EDI formulation \eqref{epsEDI} of 
$\eps$-gradient flow \eqref{epsGF} converges to the limiting EDI
\eqref{0EDI}. To this end, we need to prove three lower bounds for the 
functionals \eqref{freeEeps}, \eqref{psi}, and \eqref{c_psi} on the 
left-hand-side of \eqref{epsEDI}.  Recall the definitions of $\bar{E}, \psi, \psi^*$ in Section \ref{main}. The  lower bounds estimates are stated in the following.

\begin{thm}\label{thm1}
Assume the initial data $\rho_0^\eps$ satisfies the assumptions
of Lemma \ref{init.ass}. 
Let further $\rho_0$ be the limit of $\rho_0^\eps$ in $(\sP(\Omega), W_2)$
and $\rho_0^\eps$ be well-prepared in the sense of \eqref{id_con}.
Then   
\begin{enumerate}[(i)]
\item there exists a subsequence $\rho^\eps$ and $\rho\in C([0,T]; L^2(\Omega))$ such that \eqref{convergence_W2} holds;
\item for a.e. $t\in[0,T]$, 
the lower bound for free energy holds
\begin{equation}\label{lower1}
\liminf_{\eps \to 0} E_\eps(\rhoe) \geq \lbar{E}(\rho_t);
\end{equation}
\item for any $t\in[0,T]$, the lower bound for the dissipation on the cotangent plane holds
\begin{equation}\label{lower2}
\liminf_{\eps \to 0} \int_0^t \psi^*_\eps\left(\rho^\eps_\tau, - \frac{\delta E_\eps}{\delta \rho}(\rho^\eps_\tau)\right) \ud \tau \geq  
\int_0^t \psi^*\left(\rho_\tau, - \frac{\delta\lbar{E}}{\delta \rho}(\rho_\tau)\right) \ud \tau;
\end{equation}
\item for any $t\in[0,T]$, the lower bound for the dissipation on the tangent plane holds
\begin{equation}\label{lower3}
\liminf_{\eps \to 0} \int_0^t \psi_\eps(\rho^\eps_\tau, \pt_\tau \rho^\eps_\tau) \ud \tau \geq  \int_0^t \psi(\rho_\tau, \pt_\tau \rho_\tau) \ud \tau.
\end{equation}
\end{enumerate}
\end{thm}

As mentioned before, our approach relies on the idea of
convergence of functionals in a variational setting. In particular, we make 
use of the following result which is a special case of by now classical 
results of $\Gamma$-convergence.
See for example, \cite[Theorems 4.1, 4.4]{marcellini1978periodic}, and also
\cite{braides2006handbook, braides2002gamma, dal2012introduction} for more detailed explanations.
\begin{thm}[$\Gamma$-conv]\label{GammaMainThm}
Let $\Omega$ be an open bounded domain of $\mathbb R^n$ and 
$A_\eps(\cdot) = A(\cdot,\frac{\cdot}{\eps})$ be a symmetric positive definite
matrix. Consider the functional
\begin{equation}
{\mathcal F}_\eps(v) =\int_\Omega
\left\la A\left(x,\frac{x}{\eps}\right)\nabla v, \nabla v\right\ra\ud x,
\quad v\in H^1_0(\Omega) + w
\end{equation}
where $w\in H^1(\Omega)$ is given. Then ${\mathcal F}_\eps$ $\Gamma$-converges 
in $L^2(\Omega)$ to the following functional
\begin{equation}
{\mathcal F}(v) =\int_\Omega
\left\la \lbar{A}(x)\nabla v, \nabla v\right\ra\ud x,
\quad v\in H^1_0(\Omega) + w.
\end{equation}
In detail, 
\begin{enumerate}
\item for any $v_\eps\in H^1_0(\Omega) + w$ that converges to 
$v\in H^1_0(\Omega) + w$ in $L^2(\Omega)$, it holds that
\begin{equation}
\liminf_{\eps\to0}{\mathcal F}_\eps(v_\eps) \geq {\mathcal F}(v);
\end{equation}
\item for any $v\in H^1_0(\Omega) + w$, there exists 
$v_\eps\in H^1_0(\Omega) + w$ that converges to $v$ in $L^2(\Omega)$, such 
that 
\begin{equation}
\lim_{\eps\to0}{\mathcal F}_\eps(v_\eps) = {\mathcal F}(v).
\end{equation}
\end{enumerate}
Furthermore, the effective matrix $\lbar{A}$ can be found by the following
variational formula: for any $p\in\mathbb R^n$,
\begin{equation}\label{eff.mat}
\big\la \lbar{A}(x)p, p\big\ra 
= \inf\left\{\int_{\mathbb T^n} 
\left\la A\left(x,y\right)(\nabla v+p),\, (\nabla v + p)\right\ra\ud y,
\quad v\in H^1(\mathbb T^n)\right\}.
\end{equation}
\end{thm}
As an application, we will apply the above result to the case 
$\Omega = \mathbb T^n$ and 
\[
A(x,y) = D(x,y)\,\,(= \pi(x,y)B^{-1}(y))
\quad\text{(see \eqref{D.def}).}
\]
The resultant formula for $\lbar{A}(x)$ is given by $\lbar{D}+\lbar{G}$; 
see the expressions of $\lbar{D}$ and $\lbar{G}$ in \eqref{homog.f}. In Appendix \ref{asym.exp}, we derive the same formula using asymptotic analysis.

\begin{proof}[Proof of \eqref{lower1}] This statement follows directly
from \cite[Lemma 9.4.3]{AGS} which says that the entropy
functional is jointly lower-semicontinuous with respect to the weak
convergence of $\rhoe$ and $\pie$. In our case, it also follows simply
from the strong convergence of $f_t^\eps$ (together with the fact that
$f_t^\eps$ is uniformly bounded from above and away from zero):
\begin{eqnarray*}
\lim_{\eps\to 0}
\int_\Omega \rhoe\log\frac{\rhoe}{\pie}\,\ud x
= 
\lim_{\eps\to 0}
\int_\Omega f_t^\eps(\log f_t^\eps)\pie\,\ud x
= \int_\Omega f_t(\log f_t)\lbar{\pi}\,\ud x
= \int_\Omega \rho_t\log\frac{\rho_t}{\lbar{\pi}}\,\ud x.
\end{eqnarray*}
\end{proof}

\begin{proof}[Proof of \eqref{lower2} (time independence case)]
Let $\tau\in[0,T]$ be fixed. We will prove that
\begin{equation}\label{psi*.lsc}
\liminf_{\eps\to0}\psi^*_\eps(\rho^\eps_\tau, -\log \frac{\rho^\eps_\tau}{\pie})
\geq
\psi^*\left(\rho_\tau, -\log \frac{\rho_\tau}{\lbar{\pi}}\right).
\end{equation}
We re-write the functional $\psi^*$ in the following way,
\begin{align*}
\psi^*_\eps(\rho^\eps_\tau, -\log \frac{\rho^\eps_\tau}{\pie}) =& 
\frac12 \int_\Omega \left\la \nabla  \log \frac{\rho^\eps_\tau}{\pie}, \be \nabla \log \frac{\rho^\eps_\tau}{\pie} \right\ra \rho^\eps_\tau \ud x\\
=& 2 \int_\Omega \left\la \,  \nabla \sqrt{\frac{\rho^\eps_\tau}{\pie}}, \be\pie \nabla \sqrt{\frac{\rho^\eps_\tau}{\pie}}  \, \right\ra\ud x\\
=& 2 \int_\Omega \big\la \, \nabla w^\eps_\tau, D_\eps\nabla w^\eps_\tau 
\, \big\ra\ud x,
\end{align*}
where 
\[
w^\eps_\tau:= \sqrt{f^\eps_\tau},\,\,\,\text{and}\,\,\,
D_\eps = B^{-1}_\eps\pie.
\]
As $f_\tau^\eps\to f_\tau=\frac{\rho_\tau}{\lbar{\pi}}$ strongly
in $L^p(\Omega)$ for any $p\geq 1$, 
we have $w^\eps_\tau \to w_\tau := \sqrt{f_\tau}
=\sqrt{\frac{\rho_\tau}{\lbar{\pi}}}$ in $L^2(\Omega)$.
Now we can invoke Theorem \ref{GammaMainThm} to deduce that
\begin{eqnarray*}
&&\liminf_{\eps\to 0}2\int_\Omega
\big\la \,  \nabla w_\tau^\eps, D_\eps \nabla w_\tau^\eps
  \, \big\ra \ud x\\
&\geq& 
2\int_\Omega \la \nabla w_\tau, (\lbar{D}+\lbar{G})\nabla w_\tau \ra \ud x
=2\int_\Omega \left\la \nabla \sqrt{f_\tau}, (\lbar{D}+\lbar{G})\nabla \sqrt{f_\tau} \right\ra \ud x
\\
&=& 
2\int_\Omega \left\langle \nabla \sqrt{\frac{\rho_\tau}{\lbar{\pi}}}, (\lbar{D}+\lbar{G}) \nabla \sqrt{\frac{\rho_\tau}{\lbar{\pi}}}\right\rangle \ud x
=\frac12\int_\Omega \left\langle \nabla \log{\frac{\rho_\tau}{\lbar{\pi}}}, 
\left(\frac{\lbar{D}+\lbar{G}}{\lbar{\pi}}\right) \nabla \log{\frac{\rho_\tau}{\lbar{\pi}}}\right\rangle\rho_\tau \ud x
\\
&=&\frac12\int_\Omega \left\langle \nabla \log{\frac{\rho_\tau}{\lbar{\pi}}}, 
\lbar{B}^{-1}\nabla \log{\frac{\rho_\tau}{\lbar{\pi}}}\right\rangle\rho_\tau \ud x
= \psi^*\left(\rho_\tau, -\log \frac{\rho_\tau}{\lbar{\pi}}\right),
\end{eqnarray*}
concluding the result \eqref{psi*.lsc}, with the identification
$\lbar{B}=\left(\frac{\lbar{D}+\lbar{G}}{\lbar{\pi}}\right)^{-1}$,
from \eqref{homog.B}.

\end{proof}

\begin{proof}[Proof of \eqref{lower3} (time independence case)]
Here we establish 
\begin{equation}\label{psi.lsc0}
\liminf_{\eps\to0}\psi_\eps(\rho^\eps, s^\eps)
\geq
\psi(\rho, s)
\end{equation}
for any $\rho^\eps\wra\rho$ in $L^1(\Omega)$ and $s^\eps\wra s$ in 
$L^2(\Omega)$ with the property that
\[
f^\eps = \frac{\rho^\eps}{\pie} \longrightarrow 
f = \frac{\rho}{\lbar{\pi}}\,\,\,\text{in}\,\,\,L^2(\Omega).
\]

Using the definition of $\psi_\eps$, we have 
\begin{equation}
\begin{aligned}
\psi_\eps(\rho^\eps,s^\eps) 
=&\sup_{\xi\in L^2(\Omega)}\left\{
\int_\Omega\xi s^\eps\ud x 
- \frac12\int_\Omega\la\nabla\xi,B_\eps^{-1}\nabla\xi \ra\rho^\eps\ud x
\right\}
\end{aligned}
\end{equation}
and likewise,
\begin{equation}\label{psi.sup.def}
\begin{aligned}
\psi(\rho,s) 
=&\sup_{\xi\in L^2(\Omega)}\left\{
\int_\Omega\xi s\ud x 
- \frac12\int_\Omega\la\nabla\xi,\bar{B}^{-1}\nabla\xi \ra\rho\ud x
\right\}.
\end{aligned}
\end{equation}
Note that the supremum in both definitions can be attained. In particular, 
there is a $\tilde\xi$ such that
\begin{equation}
\begin{aligned}
\psi(\rho,s)
=&
\int_\Omega\tilde\xi s\ud x
- \frac12\int_\Omega\la\nabla\tilde\xi,\bar{B}^{-1}\nabla\tilde\xi \ra\rho\ud x
\,\,\,\text{where}\,\,\,
s = -\nabla\cdot\left(\rho\bar{B}^{-1}\nabla\tilde\xi\right).
\end{aligned}
\end{equation}

Next we make use of an approximating sequence $\tilde\xi^\eps\wra\tilde\xi$ in $H^1(\Omega)$ 
(and hence $\tilde\xi^\eps\to\tilde\xi$ in $L^2(\Omega)$) such that
\begin{equation}
\lim_{\eps\to0}
\frac12\int_\Omega\la\nabla\tilde\xi^\eps,B_\eps^{-1}\nabla\tilde\xi^\eps \ra\rho_\eps\ud x
= 
\frac12\int_\Omega\la\nabla\tilde\xi,\lbar{B}^{-1}\nabla\tilde\xi \ra \rho\ud x.
\end{equation}
The above is equivalent to 
\begin{equation}\label{approx.seq.psi}
\lim_{\eps\to0}
\frac12\int_\Omega\la\nabla\tilde\xi^\eps, D_\eps\nabla\tilde\xi^\eps \ra f^\eps\ud x
=\frac12\int_\Omega\la\nabla\tilde\xi,(\lbar{D}+\lbar{G})\nabla\tilde\xi \ra f\ud x
.
\end{equation}
The construction of $\tilde{\xi}^\eps$ can essentially be given by 
Theorem \ref{GammaMainThm} if we set $A_\eps = D_\eps f^\eps$. But in order to separate
the dependence between $D_\eps$ and $f^\eps$, a different argument is needed. We will provide the details in Appendix \ref{GammaInhomog}.

Now by the fact that 
$\tilde\xi^\eps\longrightarrow\tilde\xi$ in $L^2(\Omega)$, 
together with the assumption $s^\eps\wra s$ in $L^2(\Omega)$, we have
\[
\int_\Omega\tilde\xi^\eps s^\eps\ud x
\longrightarrow
\int_\Omega\tilde\xi s\ud x.
\]
Then \eqref{approx.seq.psi} implies that
\begin{equation}\label{psi.lsc}
\begin{aligned}
\psi(\rho,s)
=&
\int_\Omega\tilde\xi s\ud x
- \frac12\int_\Omega\la\nabla\tilde\xi,\lbar{B}^{-1}\nabla\tilde\xi \ra\rho\ud x
\\
=&
\lim_{\eps\to0}
\left\{\int_\Omega\tilde\xi^\eps s^\eps\ud x
- \frac12\int_\Omega\la\nabla\tilde\xi^\eps,{B}_\eps^{-1}\nabla\tilde\xi^\eps 
\ra\rho^\eps\ud x
\right\}\\
\leq & \liminf_{\eps\to0}\left[
\sup_{\xi}\left\{\int_\Omega\xi s^\eps\ud x
- \frac12\int_\Omega\la\nabla\xi,{B}_\eps^{-1}\nabla\xi \ra\rho^\eps\ud x
\right\}\right]
\\
\leq &\liminf_{\eps\to0} \psi(\rho^\eps, s^\eps),
\,\,\,
\end{aligned}
\end{equation}
which completes the proof for \eqref{psi.lsc}.
\end{proof}

\begin{proof}[Proof of \eqref{lower2} and \eqref{lower3}: time dependent case]
To extend the time independent case to the time dependent case and finish the 
proofs of lower bounds \eqref{lower2} and \eqref{lower3}, we will make use of a
general $\Gamma$-$\liminf$ result as stated in \cite[Cor. 4.4]{stefanelli2008brezis}.
Specifically, let $H$ be a separable and reflexive Banach space, and
$g_n$, $g_\infty: (0,T)\times H\longrightarrow(-\infty, \infty]$ be 
such that $g_n(t,\cdot)$ and $g_\infty(t,\cdot): H\longrightarrow(-\infty, \infty]$ 
are convex and for all $u\in H$ and a.e. $t\in(0,T)$, the following holds:
\begin{equation}\label{wk.normal.int}
g_\infty(t,u) 
\leq 
\inf\left\{
\liminf_{n} g_n(t,u_n): u_n\wra u\,\,\,\text{in $H$}
\right\}.
\end{equation}
Then for $p\in[1, \infty]$, $u_n\wra u$ in $L^p(0,T;H)$ 
(weak-$*$ if $p=\infty$)
and $t\longrightarrow \max\{0, -g_n(t,u_n(t)\}$ uniformly integrable, 
we have,
\begin{equation}
\int_0^T g_\infty(t,u(t))\ud t
\leq 
\liminf_n \int_0^T g_n(t,u_n(t))\ud t.
\end{equation}
Note that the uniform integrability condition is automatically satisfied if
$g_n$ are non-negative, or bounded from below. See also the remark after 
Cor. 4.4 in \cite{stefanelli2008brezis}.

For \eqref{lower2}, we set $H=H^1(\Omega)$,
\[
g^*_\eps(t,w) := 2\int_\Omega\la\nabla w, D_\eps\nabla w\ra\ud x,
\,\,\,\text{and}\,\,\,
g^*_\infty(t,w) := 2\int_\Omega\la\nabla w, (\lbar{D}+\lbar{G})\nabla w\ra\ud x.
\]
Then $g^*_\eps(t,\cdot)$ and $g^*_\infty(t,\cdot)$ are convex and
\eqref{wk.normal.int} holds true by the time independent version of
\eqref{lower2}. Hence we have
\[
\int_0^T 2\int_\Omega\la\nabla w, (\lbar{D}+\lbar{G})\nabla w\ra\ud x\ud t
\leq\liminf_\eps
\int_0^T 2\int_\Omega\la\nabla w^\eps(t), D_\eps\nabla w^\eps(t)\ra\ud x\ud t
\]
provided $w^\eps\wra w$ in $L^2((0,T);H)$. This last condition is satisfied
by the identification $w^\eps(t) = \sqrt{\frac{\rho^\eps(t)}{\pie}}$ and 
$w(t) = \sqrt{\frac{\rho(t)}{\lbar{\pi}}}$, and \eqref{f.st.L2L2}.
This concludes the lower bound \eqref{lower2}.

For \eqref{lower3}, we set $H = L^2(\Omega)$,
\[
g_\eps(t,s) := \psi_\eps(\rho(t),s) = 
\frac12\int_\Omega\la\nabla u^\eps, B_\eps^{-1}\nabla u^\eps\ra\rho^\eps(t)\ud x
\,\,\,\text{with}\,\,\,
-\nabla\cdot (\rho^\eps B_\eps^{-1}\nabla u^\eps) = s,
\]
and
\[
g_\infty(t,s) := \psi(\rho(t),s) = 
\frac12\int_\Omega\la\nabla u, \lbar{B}^{-1}\nabla u\ra\rho(t)\ud x
\,\,\,\text{with}\,\,\,
-\nabla\cdot (\rho \lbar{B}^{-1}\nabla u) = s.
\]
Again, $g_\eps(t,\cdot)$ and $g_\infty(t,\cdot)$ are convex 
because the map $s\to u^\eps$ or $u$ is uniquely defined and linear.
By \eqref{psi.lsc0}, \eqref{wk.normal.int} is satisfied. Hence, we have
\[
\int_0^T\psi\big(\rho(t),s(t)\big)\ud t 
\leq
\liminf\int_0^T\psi_\eps\big(\rho^\eps(t), s^\eps(t)\big)\ud t
\]
upon the identification 
$s^\eps(t) = \partial_t\rho^\eps_t$ and 
$s(t) = \partial_t\rho_t$. The fact that $s^\eps\wra s$ in $L^2((0,T);H)$ follows
from \eqref{conv5}. Lower bound \eqref{lower3} is thus proved.

The above conclude the proof for Theorem \ref{thm1}.
\end{proof}

\section{Comparison between limiting Wasserstein distances}\label{sec5}

In this section, we use the just established convergence result for gradient 
flows in EDI form to  further analyze the induced limiting Wasserstein distance 
$\lbar{W}$. In particular, 
we will show that the limiting Wasserstein metric $\lbar{W}$ is in
general, different, 
and in fact {\em strictly larger than} $W_{\text{GH}}$
obtained from the Gromov-Hausdorff limit of $W_\eps$ which is a commonly
considered mode of convergence of metric spaces. 
 Gromov-Hausdorff distance can be used to compare the distortion of two 
metric spaces from being isometric. 
The particular property needed in this
paper is that the Gromov-Hausdorff convergence of a compact metric space $\Omega_k$ implies the Gromov-Hausdorff convergence of the Wasserstein space  $(\sP(\Omega_k), W_k)$ \cite[Theorem 28.6]{villani2009optimal}.    
Briefly stated, let $({\mathcal X}, d_{\mathcal X})$ and 
$({\mathcal Y}, d_{\mathcal Y})$ be two metric spaces.
Their Gromov-Hausdorff distance is defined as 
\cite[(27.2)]{villani2009optimal}
\begin{equation}\label{GHDist}
D_{GH}({\mathcal X}, {\mathcal Y}) = \frac12\inf_{\mathcal R}
\sup_{(x,y),(x',y')\in {\mathcal R}}\Big|
d_{\mathcal X}(x,x')-
d_{\mathcal Y}(y,y')
\Big|,
\end{equation}
where ${\mathcal R}\subset {\mathcal X}\times{\mathcal Y}$ is a 
correspondence or relation between
${\mathcal X}$ and ${\mathcal Y}$.
We refer to 
\cite[Chapters 27, 28]{villani2009optimal} for more detailed information about 
the concept of Gromov-Hausdorff distances and convergence. 
For our application, we will take
$({\mathcal X}, d_{\mathcal X}) := (\Omega, d_\eps)$ or 
$(\sP(\Omega), W_\eps)$.

We remark that several of the following statements require the
existence of densities (with respect to Lebesgue measure) for the 
underlying probability measures and the space to be geodesic complete. 
These are automatically satisfied by our standing assumptions (see Section \ref{main}).

\subsection{Effective Wasserstein distance $\lbar{W}$ induced by convergence 
of gradient flows}\label{WepsWbar}
For convenience, we recall here the Kantorovich and Benamou-Brenier 
formulations \eqref{Wc} and \eqref{WBB} for our $\eps$-Wasserstein metric 
$W_\eps$:
\begin{equation}\label{W_eps_5}
W_\eps^2(\rho_0, \rho_1):=\inf\left\{
\iint d^2_\eps(x,y) \ud \gamma(x,y); \quad 
\int_{\Omega} \gamma( x,\ud y)   = \rho_0(x)\ud x, \,\, 
\int_{\Omega} \gamma(\ud x,  y)   = \rho_1(y)\ud y
\right\}
\end{equation}
and 
\begin{equation}\label{WBB-2}
W_\eps^2(\rho_0, \rho_1) := \inf\left\{
\int_0^1\int \rho_t(x)\langle B_\eps(x) v_t(x), v_t(x)\rangle \ud x\ud t,\quad
(\rho_t, v_t)\in V(\rho_0, \rho_1)
\right\},
\end{equation}
where $V$ is defined in \eqref{ContEqn}. The $\eps$-metric $d_\eps$
on $\Omega\subset \bR^n$ is given via the least action  
\begin{equation}\label{eps-d}
d_\eps^2(x,y) := \inf\left\{ \int_0^1 \la B_\eps(z_t) \dot{z}_t, \dot{z}_t \ra \ud t, \quad z_0 = x, \quad z_1=y \right\}.
\end{equation}
A curve $z(\cdot)\in AC([0,1]; \bR^n)$ that achieves the infimum in 
\eqref{eps-d} is a geodesic in the metric space $(\bR^n, d_\eps).$ 
From \cite[Theorem A,B]{bernard2007optimal}, \eqref{W_eps_5} and \eqref{WBB-2}
are equivalent. 

The same formulations hold for our induced limit Wasserstein
distance $\lbar{W}$. More precisely, we have
\begin{equation}\label{W_*}
\lbar{W}^2(\rho_0, \rho_1):= \inf\left\{
\int \int \lbar{d}^2(x,y) \ud \gamma(x,y); \quad 
\int_{\Omega} \gamma( x,\ud y)   = \rho_0(x)\ud x, \,\, 
\int_{\Omega} \gamma(\ud x, y)   = \rho_1(y)\ud y
\right\},
\end{equation}
and the equivalent formulation 
\begin{equation}
\lbar{W}^2(\rho_0, \rho_1):= \inf\left\{
\int_0^1\int \rho_t(x)\langle \lbar{B}  v_t(x), v_t(x)\rangle \ud x\ud t,\quad
(\rho_t, v_t)\in V(\rho_0, \rho_1)
\right\}.
\end{equation}
Here the constant matrix $\lbar{B}$ is defined in  \eqref{homog.B}
and the induced-metric $\lbar{d}$ on $\Omega\subset \bR^n$ is again
given via the least action   
\begin{equation}\label{d_*}
\lbar{d}^2(x,y) := \inf\left\{ \int_0^1 \la \lbar{B} \dot{z}_t, \dot{z}_t \ra \ud t, \quad z_0 = x, \quad z_1=y \right\}.
\end{equation}
From the Euler-Lagrangian equation for the minimizer of \eqref{d_*}, the 
optimal curve $\tilde{z}(\cdot)$ that achieves the least action satisfies
$\lbar{B}\ddot{\tilde{z}}_t=0$, and hence it has   constant speed,
$\dot{\tilde{z}}_t = y-x$. Thus we have explicitly
\begin{equation} 
\lbar{d}^2(x,y) = \la \lbar{B} (y-x), y-x \ra = \la \lbar{B} \hat{n}, \hat{n} \ra |y-x|^2,
\quad\text{where}\quad
\hat{n}=\frac{y-x}{|y-x|}.
\end{equation}

Note that both $W^\eps$ and $\lbar{W}$ induce a Riemannian metric on 
$\sP(\Omega)$. More precisely, for any $\rho\in\sP(\Omega)$, and any
$s_1, s_2\in T_{\sP}$, the tangent plane at $\rho$,
the first fundamental form are defined respectively as
\begin{eqnarray}
\big\la s_1, s_2\big\ra_{T_\sP, T_\sP, \eps}&:=& 
\int  \rho(x)\langle B^{-1}_\eps(x) \nabla u_1(x), \nabla u_2(x)\rangle \ud x,
\end{eqnarray}
where
$s_i = -\nabla \cdot (\rho B^{-1}_\eps\nabla u_i)$, $i=1,2$
for $W^\eps$, and
\begin{eqnarray}
\big\la s_1, s_2\big\ra_{T_\sP, T_\sP}&:=& \int  \rho(x)\langle \lbar{B}^{-1}(x) \nabla u_1(x), \nabla u_2(x)\rangle \ud x, 
\label{limitRe}
\end{eqnarray}
where
$s_i = -\nabla \cdot (\rho \lbar{B}^{-1}\nabla u_i)$, $i=1,2$
for $\lbar{W}$.
This is also manifested by the fact that both the corresponding
dissipation functionals are {\em bilinear forms} in $s$:
\[
\psi_\eps(\rho, s)= \frac12 \int_\Omega \la \nabla u, B_\eps^{-1} \nabla u \ra \rho \ud x  \quad \text{ with }\, s=-\nabla \cdot \bbs{\rho B_\eps^{-1} \nabla u},
\]
and
\[
\psi(\rho, s)= \frac12 \int_\Omega \la \nabla u, \lbar{B}^{-1} \nabla u \ra \rho \ud x  \quad \text{ with }\, s=-\nabla \cdot \bbs{\rho \lbar{B}^{-1} \nabla u}.
\]

\subsection{The Gromov-Hausdorff limit $W_{\text{GH}}$ of $W_\eps$}
\label{sec:GH}
Now we consider the convergence in the Gromov-Hausdorff sense
of $W_\eps$ to a limiting Wasserstein metric, denoted as $W_{\text{GH}}$.

We first show that even in one dimension, in general it is always the
case that $W_{\text{GH}}< \lbar{W}$ unless $\pi_\eps$ and $B_\eps$ are
related to each other in some specific way.
Recall the metric $d_\eps$ in \eqref{eps-d}. From the Euler-Lagrangian equation for the minimizer $z_t = \tilde{z}^\eps_t$, we have
$$
\frac{\ud}{\ud t} (2 B_\eps(z_t)\dot{z}_t) 
= B_\eps'(z_t)(\dot{z}_t)^2,
$$
leading to $B_\eps'(z_t)\dot{z}_t^2+ 2 B_\eps(z_t)\ddot{z}_t=0$ and thus 
$$
B_\eps(z) \dot{z}^2=C_\eps(x,y),
\quad\text{for some constant $C_\eps(x,y)$.}
$$
Upon solving this ODE for $z_t$ with the two boundary conditions 
$z(0)=x$, $z(1)=y$, we have
$$
\sqrt{C_\eps(x,y)}=\int_x^y \sqrt{B_\eps(z)} \ud z.
$$
Hence the infimum in \eqref{eps-d} is given by
\begin{equation}
d^2_\eps(x,y) = C_\eps(x,y) = \left(\int_x^y \sqrt{B_\eps(z)} \ud z\right)^2.
\end{equation} 
As $B_\eps(x)=B(\frac{x}{\eps})$,
it is easy to verify that for any $x,y\in \Omega$,  there exist  some integer $N_\eps$ and $\delta\in(-1,1)$, such that $y-x=N_\eps\eps + \delta \eps$ and $N_\eps \eps \to |x-y|$.  Notice also  $B(\cdot)$ is $1$-periodic. Hence
\begin{eqnarray*}
d^2_\eps(x,y) = \bbs{ \eps \int_{\frac{x}{\eps}}^{\frac{y}{\eps}} \sqrt{B(s)} \ud s}^2 &=& \left(\eps N_\eps \int_0^1 \sqrt{B(s)} \ud s + \eps \int_0^\delta \sqrt{B(s)} \ud s\right)^2\\
&\stackrel{\lra}{\eps\to0} & 
|x-y|^2\left(\int_0^1 \sqrt{B(s)} \ud s\right)^2
=: d^2_{\text{GH}}(x,y).
\end{eqnarray*}
Notice that if one chooses ${\mathcal R}$ to be the identity map as
the correspondence between the metric spaces 
$\mathcal{X}:=(\Omega, d_\eps)$ and $\mathcal{Y}:= (\Omega, d_{\text{GH}})$, 
then from \eqref{GHDist}, we have
$$D_{\text{GH}}(\mathcal{X},\mathcal{Y})\leq \frac{1}{2} \sup_{(x,x), (y,y)\in \mathcal{X}\times\mathcal{Y}}|d_\eps(x,y)-d_{\text{GH}}(x,y)| \to 0.$$
Hence the one dimensional metric space $(\Omega, d_\eps)$ 
Gromov-Hausdorff converges to $(\Omega, d_{\text{GH}})$.
By \cite[Theorem 28.6]{villani2009optimal}, 
the Wasserstein distance $W_\eps$ defined in \eqref{W_eps_5} also
converges to the following limiting Wasserstein distance $W_{\text{GH}}$ 
in the Gromov-Hausdorff sense, 
 \begin{equation}\label{W_GH}
 W_{\text{GH}}^2(\rho_0, \rho_1):= \inf\left\{
\int \int d^2_{\text{GH}}(x,y) \ud \gamma(x,y); \quad \int_{\Omega} \gamma( x,\ud y)   = \rho_0(x) \ud x, \,\, \int_{\Omega} \gamma(\ud x, y)   = \rho_1(y) \ud y
\right\}.
 \end{equation}
 Again by \cite[Theorem AB]{bernard2007optimal}, $W_{\text{GH}}$ can be equivalently written in the Benamou-Brenier formulation
 \begin{equation}
W_{\text{GH}}^2(\rho_0, \rho_1):= \inf\left\{
\int_0^1\int \rho_t(x)\langle \lbar{C}  v_t(x), v_t(x)\rangle \ud x\ud t,\quad
(\rho_t, v_t)\in V(\rho_0, \rho_1)
\right\} 
\end{equation} 
with $\displaystyle \lbar{C}=\left(\int_0^1 \sqrt{B(s)} \ud s\right)^2.$  

On the other hand, in one dimension, we can solve the cell problem 
\eqref{cell1} explicitly:
\begin{eqnarray*}
\partial_y\big(D(x,y)\partial_y w(x,y)\big) 
&=& -\partial_y\left(D(x,y)\right),
\quad \text{where}\,\,\,D(x,y) = \pi(x,y)B(y)^{-1},\\
\partial_y w(x,y) 
&=& -1 + \frac{C(x)}{D(x,y)} 
\quad \text{with}\,\,\,C(x) = \left(\int \frac{1}{D(x,y)}\ud y\right)^{-1}.
\end{eqnarray*}
Then \eqref{homog.coeff} and \eqref{homog.B} are given as
\begin{eqnarray*}
\lbar{D}(x) &=& \int D(x,y)\ud y,\\
\lbar{G}(x) 
&=& \int D(x,y) \left(-1 + \frac{C(x)}{D(x,y)}\right)\ud y
= -\int D(x,y)\ud y  + 
\left(\int \frac{1}{D(x,y)}\ud y\right)^{-1},\\
\lbar{B}
&=&
\left(\frac{\lbar{D} + \lbar{G}}{\lbar{\pi}}\right)^{-1}
= \lbar{\pi}\int\frac{1}{D(x,y)}\ud y
= \lbar{\pi}\int\frac{B(y)}{\pi(x,y)}\ud y.
\end{eqnarray*}

By the Cauchy-Schwarz inequality,  we always have
\begin{align*}
\lbar{C}
=\left(\int_0^1 \sqrt{B(s)} \ud s\right)^2 
=&\left(\int_0^1 \sqrt{\pi(x,y)}\sqrt{\frac{B(y)}{\pi(x,y)}} \ud y\right)^2 
\\
\leq&  
\left(\int \pi(x,y)\ud y\right)
\left(\int \frac{B(y)}{\pi(x,y)}\ud y\right)
= \lbar{B}(x),
\end{align*}
and the equality holds if and only if there exists some constant $c>0$ such that
\begin{eqnarray}\label{==}
\sqrt{\pi(x,y)} = c\sqrt{\frac{B(y)}{\pi(x,y)}},
\quad\text{i.e.}\quad
    \pi(x,y)=\pi(y) = c\sqrt{B(y)}.
\end{eqnarray}
Hence, unless $\pi(y)=c\sqrt{B(y)}$, we always have
\[
d_{\text{GH}}(x,y) < \lbar{d}(x,y)\quad\text{for all $x,y\in\Omega$}
\]
i.e. $W_{\text{GH}} < \lbar{W}.$ As an afterthought, it seems not quite surprising  that some condition, such as \eqref{==}, is needed in order for $\lbar{W}$ to be equal to $W_{\text{GH}}$. We will elaborate upon this at the end of this section.

Next, we illustrate the $n$-dimensional case by means of an example.
From \cite[Section 3.3]{braides2002gamma}, it is shown that the functional
\begin{equation}
{\mathcal F}_\eps(z) = 
\int_0^1 \la B_\eps(z_t) \dot{z}_t, \dot{z}_t \ra \ud t, 
\quad \text{for} \quad z(\cdot)\in (H^1([0,1]))^n  
\quad \text{with}\quad z_0 = x, \quad z_1=y,
\end{equation}
$\Gamma$-converges with respect to the strong $L^2(0,1)$-topology to
\begin{equation}
{\mathcal F}(z) = \int_0^1 \varphi(\dot{z}(t)) \ud t
\quad \text{for}\quad z(\cdot)\in (H^1([0,1]))^n, 
\quad \text{with}\quad z_0 = x, \quad z_1=y,
\end{equation}
where the limiting integrand $\varphi$ is given by
\begin{equation}\label{varphi.limit.rep}
\varphi(v) := \lim_{T\to +\8} \inf_{u\in (H^1_0([0,T]))^n}
\left\{\frac{1}{T} \int_0^T \la B(u(t)+vt) (\dot{u}(t)+v), \, \dot{u}(t)+v \ra \ud t\right\}.
\end{equation}

Now following \cite[Example 3.3]{braides2002gamma}, we consider
$B_\eps(z)=b(\frac{z}{\eps})$ where $b$ is the following $1$-periodic
function on $[0,1]^n$,
\begin{equation*}
b(y) =
\left\{
\begin{array}{ll}
\beta & \text{ if } y\in (0,1)^n;\\
\alpha & \text{ if for some }i,\,  y_i\in \bZ.
\end{array}
\right.
\end{equation*}
If $n \alpha < \beta$, one obtains that the limiting energy integrand
$\varphi$ is given by 
\begin{equation}\label{varphi.finsler}
\varphi(v)=\alpha\left(\sum_{i=1}^n |v_i|\right)^2.
\end{equation}
Using the property of $\Gamma$-convergence 
\cite[Theorem 1.21]{braides2002gamma}, we deduce also the convergence of 
the minimum value $d^2_\eps$ of $\mathcal{F}_\eps$ to the minimum value
$d^2_{\text{GH}}$ of $\mathcal{F}$, where 
\begin{equation}\label{thdd}
d_{\text{GH}}(x,y)=  \sqrt{\alpha}\left(\sum_{i=1}^n |\hat{n}_i|\right)|y-x| = \sqrt{\alpha}\|y-x\|_{\ell^1}  
\,\,\,  \text{with $\hat{n}=\frac{y-x}{|y-x|}.$}
\end{equation}

On the other hand, note that the value $\alpha$ is attained only on 
the $(n-1)$-dimensional set $\bigcup_{i=1}^n\{y_i \in \mathbb Z\}$.
This set is {\em invisible} by $\lbar{B}$ which is obtained by solving
the elliptic cell problem \eqref{cell1}. 
Hence the induced limiting Wasserstein distance $\lbar{W}$ \eqref{W_*} 
with $\lbar{d}$ defined in \eqref{d_*} is 
$\lbar{d}(x,y)=\beta|x-y|$ for all $x,y\in\Omega$.  
Thus, for this example, we have
\[
d_{\text{GH}}(x,y)=  \sqrt{\alpha} \|y-x\|_{\ell^1} 
\leq \sqrt{\alpha n}\|y-x\|_{\ell^2} 
< \sqrt{\beta} |y-x| = \lbar{d}(x,y).
\]
Hence we have again $W_{\text{GH}} < \lbar{W}$.

We would like to point out that for the above example, the integrand
$\varphi$ in \eqref{varphi.limit.rep} is always quadratic, or homogeneous of 
degree 2 in $p$. (In fact, for any 
$\lambda \neq 0$, by applying the change of variables 
$\tilde{t}=\lambda t, \tilde{u}(\tilde{t})=u(t)$, it is easy to verify that
$\varphi(\lambda v) = \lambda^2 \varphi(v)$.) However, the
$\varphi$ in \eqref{varphi.finsler} is {\em not bilinear} in $p$, in contrast to the $\varphi$ in 
\eqref{d_*}:
\[
\varphi(p) = \la \bar{B} p, p \ra.
\]

Below we give further remarks about the discrepancy between $\lbar{W}$ and $W_{\text{GH}}$.
\begin{enumerate}
\item
We first explain the condition \eqref{==}. This is nothing but the fact that one can choose the Riemannian metric $(\bR, g_\eps)$ with $(g_\eps)_{ij}(x)= B_\eps(x),$ so that the Wasserstein distance on $(\bR, g_\eps)$ coincides with $W_\eps$. 
More precisely, the condition \eqref{==} implies the volume form on $(\bR, g_\eps)$ is
\begin{equation}
    \ud\text{Vol} = \sqrt{|g_\eps|}\ud x = \sqrt{B_\eps}\ud x = c\pi_\eps(x)\ud x=c\pi(\frac{x}\eps)\ud x.
\end{equation}
Therefore, the heat flow on $(\bR, g_\eps)$, in terms of the density function with respect to the volume element $\ud\text{Vol}$ is given by
\begin{equation}
    \pt_t p_\eps = \frac{1}{\sqrt{|g_\eps|}}\nabla\cdot(\sqrt{|g_\eps|}g^{ij}_\eps\nabla p_\eps)= \frac{1}{\pi_\eps} \nabla \cdot (\pi_\eps \be \nabla p_\eps).
\end{equation}
This equation, in terms of the density function $\rho_\eps(x,t) = p_\eps(x,t) \sqrt{|g_\eps|} = p_\eps(x,t) \pi_\eps(x)$, is exactly the $W_\eps$-gradient flow with respect to the relative entropy $E_\eps$ in \eqref{freeEeps}:
\begin{eqnarray}
    \pt_t \rho_\eps = \nabla \cdot (\pi_\eps \be \nabla \frac{\rho_\eps}{\pi_\eps})=\nabla \cdot \bbs{ \rho^\eps_t B_\eps^{-1} \nabla \frac{\delta E_\eps}{\delta \rho}(\rho^\eps_t) }.
\end{eqnarray}
Therefore, condition \eqref{==} means that the discrepancy between $\overline{W}$ and $W_{\text{GH}}$ does not happen in one dimension when one considers homogenization of heat flow on $(\bR, g_\eps)$. In other words, the homogenized  heat flow in one dimension  naturally induces the same limiting distance as  finding the limiting minimum path on $(\bR, g_\eps)$. On the other hand, even in one dimension, the convergence of the discrete transport distance to continuous transport distance $W_2$ requires an isotropic mesh condition \cite[eq. (1.3)]{gladbach2020homogenisation}. Without this condition, the discrete-to-continuous limiting  distance in the Gromov-Hausdorff sense can be different from the continuous transport distance $W_2$ \cite[Theorem 1.1, Remarks 1.2 and 1.3]{gladbach2020homogenisation}.

\item
We believe that the above conclusion of $W_{\text{GH}} < \lbar{W}$ is true in general, particularly in higher dimensions, even if we consider heat flow. This is because
the Gromov-Hausdorff limit $d_{\text{GH}}$ of $d_\eps$ involves finding the 
minimum or geodesic distance between two points as indicated in \eqref{eps-d}. 
This amounts to searching for the {\em minimum path} in the 
underlying spatial inhomogeneity.
On the other hand, the $\lbar{B}$ in the limiting
induced distance $\lbar{d}$ is found by solving an elliptic cell-problem 
\eqref{cell1} which 
requires taking some {\em average} of the spatial inhomogeneity.
(Note that in contrast, in one dimension, any path will explore the whole inhomogeneous landscape.)
Hence, in general $d_{\text{GH}}$ and $W_{\text{GH}}$ should be smaller 
than $\lbar{d}$ and $\lbar{W}$.
See also the discussion in 
\cite[p.~4298]{forkert2022evolutionary} and the work 
\cite{gladbach2020homogenisation}.
\end{enumerate}

\section{Conclusion}
This paper provides a variational framework using the energy dissipation inequality
(EDI) to prove the convergence of gradient flows in Wasserstein spaces.
Our key contribution is the incorporation of fast oscillations in the 
underlying energy and medium. In particular, the gradient-flow structure is 
preserved in the limit but is described
with respect to an effective energy and metric. Our result is
consistent with asymptotic analysis from the realm of homogenization. Even
though we apply the result to a linear Fokker-Planck equation in a
continuous setting, we believe the approach is applicable to a broader class 
of problems including nonlinear equations or evolutions on graphs
and networks.

\subsection*{Competing Interests}
The authors declare none.

\subsection*{Acknowledgements}
Yuan Gao was partially supported by NSF DMS-2204288 and  NSF CAREER DMS-2440651. 
We would also like to thank the anonymous referees for their constructive comments.

\appendix

\section{Asymptotic analysis for the $\eps$-gradient flow}\label{asym.exp}
In this section, we use the method of asymptotic expansion to analyze
the convergence of the $\eps$-Fokker-Planck equation \eqref{epsFP} 
(or \eqref{epsGF}) to the limiting homogenized one \eqref{0GF}.

Recall the assumptions \eqref{Bform} and \eqref{piform}
for $B_\eps$ and $\pie$ in Section \ref{main} and the definition of
fast variable $\displaystyle y:= \frac{x}{\eps}$. 
Introducing 
\begin{equation}\label{D.def}
D(x,y)=\pi(x,y)B^{-1}(y),
\end{equation}
then \eqref{backward} reads
\begin{equation}\label{tmf}
\pt_t f^\eps =\frac{1}{\pie} \nabla \cdot \bbs{   D(x,\frac{x}{\eps}) \nabla f^\eps}.
\end{equation}

Consider  the ansatz
\begin{equation}
f^\eps\big(x,t\big) = f_0\big(x,\frac{x}{\eps},t\big) + \eps f_1\big(x,\frac{x}{\eps},t\big) + O(\eps^2) \quad
\text{with $f_0$ and $f_1$ $1$-periodic in $y$.}
\end{equation}
Substituting it into \eqref{tmf}, we have
\begin{align}\label{feps}
\pt_t 
\big(f_0+\eps f_1 + O(\eps^2)\big)
= \frac{1}{\pi(x,y)}\left(\nabla_x + \frac{1}{\eps} \nabla_y\right)\cdot \bbs{D(x,y)\left(\nabla_x+\frac{1}{\eps}\nabla y\right)
\big(f_0+\eps f_1 + O(\eps^2)\big)}.
\end{align} 
Terms of different orders are analyzed as follows.

\begin{description}
\item[(I) $\frac{1}{\eps^2}$-terms] They satisfy,
\begin{align*}
\nabla_y \cdot \bbs{D(x,y) \nabla_y f_0(x,y,t)} =0.
\end{align*}
Multiply the above by $f_0(x,y,t)$ and then integrate over $y$ gives
$\displaystyle \int |\nabla_y f_0(x,y,t)|^2\ud y = 0$ 
which implies $f_0(x,y,t)=f_0(x,t).$

\item[(II) $\frac{1}{\eps}$-terms] They satisfy,
\begin{equation}
\nabla_y \cdot \bbs{D(x,y) (\nabla_x f_0 + \nabla_y f_1)}=0.
\end{equation} 
For $i=1,2,\ldots d$, let $w_i(y)$ be the solution to the cell problem
\begin{equation}\label{cell1}
\nabla_y \cdot \bbs{D(x,y)\nabla_y w_i(x,y)} +  \nabla_y \cdot \bbs{D(x,y) \vec{e}_i}=0,
\end{equation}
where $\vec{e}_i$ is the unit vector in $i$-coordinate. 
The above equation is solvable for each $i$ due to the compatibility condition
$\displaystyle \int \nabla_y \cdot \bbs{D(x,y) \vec{e}_i} \ud y = 0$.
Then we can write $f_1$ as
$$
f_1(x,y,t)=\sum_i \pt_{x_i} f_0(x,t) w_i(x,y).
$$

\item[(III) $O(1)$-terms] Collecting the $O(1)$-terms in \eqref{feps} and
integrating with respect to $y$ lead to
 \begin{align*}
 \pt_t f_0(x,t)\,  \bar{\pi}(x) 
= \nabla_x \cdot (\overline{D}(x) \nabla_x f_0(x,t)) + \nabla \cdot 
\bbs{\sum_i \pt_{x_i} f_0(x,t) \overline{G}_i(x) },
\end{align*}
where
\begin{equation}\label{homog.coeff}
\overline{D}(x):=\int \pi(x,y) B^{-1}(y) \ud y, \quad  
\lbar{G}_i(x):= \int \pi(x,y) B^{-1}(y) \nabla_y w_i(x,y) \ud y, 
\end{equation}
and $\displaystyle \lbar{\pi} = \int\pi(x,y)\ud y$; see \eqref{pi.ave}.
\end{description}

Then the leading dynamics in terms of $f_0$ is given by
\begin{equation}\label{homog.f}
\pt_t f_0 = \frac{1}{\overline{\pi}} 
\nabla \cdot \bbs{(\overline{D}+\lbar{G}) \nabla f_0},
\quad\text{where}\,\,\,\lbar{G} = (G_1, G_2,\ldots G_n).
\end{equation} 
Upon defining
\begin{equation}\label{homog.B}
\lbar{B}(x) = \left(\frac{\lbar{D}+\lbar{G}}{\lbar{\pi}}\right)^{-1},
\end{equation}
in terms of $\rho=f_0 \lbar{\pi}$, \eqref{homog.f} can be written as
 \begin{equation}\label{homog.rho}
 \pt_t \rho = \nabla \cdot \bbs{\rho\, \overline{B}^{-1} \nabla \log 
\frac{\rho}{\lbar{\pi}}}.
\end{equation}  
 
The above procedure certainly works for the simpler uniform convergence
case $\pie=\pie^{\text{II}}$ in \eqref{pie} which converges uniformly to 
$\pi_0$. We find it illustrative to write down the homogenized limit equation.
In this case, the definition of $D$ \eqref{D.def}, the cell problem 
\eqref{cell1} and the effective coefficients \eqref{homog.coeff} now become
\[
D(x,y) = \pi_0(x)B^{-1}(y),\quad
\nabla_y \cdot \bbs{B^{-1}(y)\nabla_y w_i(y)} 
+ \nabla_y \cdot \bbs{B^{-1}(y) \vec{e}_i}=0,
\]
and
\[
\overline{D}(x):=\pi_0(x)\int  B^{-1}(y) \ud y, \quad  
\lbar{G}(x):= \pi_0(x) \int  B^{-1}(y) \nabla_y w(y) \ud y,
\quad(\text{where}\,\,\,w=(w_1,w_2\ldots w_n)),
\]
so that
\begin{equation}\label{homog.B2}
\lbar{B}(x) 
= \left(\frac{\lbar{D}(x) + \lbar{G}(x)}{\pi_0(x)}\right)^{-1}
= \left(\int  B^{-1}(y) \ud y + \int  B^{-1}(y) \nabla_y w(y) \ud y\right)^{-1}.
\end{equation}
Then the effective Fokker-Planck equation is given by
\begin{equation}
\begin{aligned}
\pt_t \rho 
= \nabla \cdot \bbs{ \rho\, \overline{B}^{-1} \nabla \log \frac{\rho}{\pi_0}}.
\end{aligned}
\end{equation}
Comparing \eqref{homog.B} and \eqref{homog.B2}, it is clear 
that there is interaction between $B_\eps$ and $\pie$ in the former case 
but not in the latter. 

\section{Construction of $\tilde{\xi}^\eps$ for \eqref{approx.seq.psi}}\label{GammaInhomog}

Here we construct an approximating sequence $\tilde\xi^\eps\wra\tilde\xi$ in $H^1(\Omega)$ such that \eqref{approx.seq.psi} holds. As mentioned, due to the spatially varying weight function
$f^\eps$, in order to decouple the dependence between $D_\eps$ and $f^\eps$, an extra step is needed if we want to invoke the classical $\Gamma$-convergence result Theorem \ref{GammaMainThm}.
Without loss of generality, we assume that $\tilde{\xi}$ is smooth so that pointwise evaluation $\tilde{\xi}(x)$ is well-defined. This can be achieved by first convolving
$\tilde{\xi}$ with a smooth kernel. We also recall by statement (1) of Lemma \ref{lem_reg} that $f$ is a bounded and uniformly positive function.

For this purpose, we write for any $\tilde\xi^\eps$ that
\begin{eqnarray*}
&&\frac12\int_\Omega\la\nabla\tilde\xi^\eps, D_\eps\nabla\tilde\xi^\eps \ra f^\eps\ud x\\
& = & 
\frac12\int_\Omega\la\nabla\tilde\xi^\eps, D_\eps\nabla\tilde\xi^\eps \ra f_c\ud x
+ \frac12\int_\Omega\la\nabla\tilde\xi^\eps, D_\eps\nabla\tilde\xi^\eps \ra (f-f_c)\ud x
+\frac12\int_\Omega\la\nabla\tilde\xi^\eps, D_\eps\nabla\tilde\xi^\eps \ra (f^\eps-f)\ud x,
\end{eqnarray*}
where $f_c$ is some continuous function approximating $f$. 
Next, we partition $\Omega$ into finitely many cubes $C_j$ and define the following
piece-wise constant function
\[
\bar{f}_c(x) = \bar{f}_{c_j} := \frac{1}{|C_j|}\int_{C_j}f_c \ud x \quad
\text{for $x\in C_j$}.
\]
Hence
\[
\frac12\int_\Omega\la\nabla\tilde\xi^\eps, D_\eps\nabla\tilde\xi^\eps \ra f_c\ud x
= 
\sum_j\frac12\int_\Omega\la\nabla\tilde\xi^\eps, D_\eps\nabla\tilde\xi^\eps \ra \bar{f_c}_j\ud x 
+
\sum_j\frac12\int_\Omega\la\nabla\tilde\xi^\eps, D_\eps\nabla\tilde\xi^\eps \ra 
(f_c-\bar{f_c}_j)\ud x.
\]
With the above, we have
\begin{eqnarray*}
&&\lim_{\eps\to0}\frac12\int_\Omega\la\nabla\tilde\xi^\eps, D_\eps\nabla\tilde\xi^\eps \ra f^\eps\ud x\\
& = &
\lim_{\eps\to0}
\frac12\int_\Omega\la\nabla\tilde\xi^\eps, D_\eps\nabla\tilde\xi^\eps \ra \bar{f}_c\ud x 
+\lim_{\eps\to0}
\frac12\int_\Omega\la\nabla\tilde\xi^\eps, D_\eps\nabla\tilde\xi^\eps \ra (f_c-\bar{f}_c)\ud x\\
&&
+ \lim_{\eps\to0}\frac12\int_\Omega\la\nabla\tilde\xi^\eps, D_\eps\nabla\tilde\xi^\eps \ra (f-f_c)\ud x
+\lim_{\eps\to0}\frac12\int_\Omega\la\nabla\tilde\xi^\eps, D_\eps\nabla\tilde\xi^\eps \ra (f^\eps-f)\ud x.
\end{eqnarray*}

Now on each $C_j$, we can invoke Theorem \ref{GammaMainThm} to state the existence of recovery sequence
$\tilde\xi^\eps_j\wra\tilde\xi$ in $H^1_0(C_j) + {g_c}_j$, where
${g_c}_j = \tilde\xi\Big|_{\partial C_j}$ such that
\begin{equation}\label{GammaConstMult1}
\lim_{\eps\to0}
\frac12\int_{C_j}\la\nabla\tilde\xi^\eps_j, D_\eps\nabla\tilde\xi^\eps_j \ra \bar{f_c}_j\ud x
=\frac12\int_{C_j}\la\nabla\tilde\xi,(\lbar{D}+\lbar{G})\nabla\tilde\xi \ra\bar{f_c}_j\ud x.
\end{equation}
Now let $\tilde\xi^\eps = \tilde\xi^\eps_j$ on $C_j$. Note that $\tilde\xi^\eps$ thus
defined is a global $H^1$-function on $\Omega$. As there are only finitely many cubes $C_j$, we can conclude that
\begin{equation}\label{GammaConstMult2}
\lim_{\eps\to0}
\frac12\int_{\Omega}\la\nabla\tilde\xi^\eps, D_\eps\nabla\tilde\xi^\eps \ra \bar{f_c}\ud x
=\frac12\int_{\Omega}\la\nabla\tilde\xi,(\lbar{D}+\lbar{G})\nabla\tilde\xi \ra \bar{f_c}\ud x.
\end{equation}
Hence we have
\begin{eqnarray}
&&\lim_{\eps\to0}\frac12\int_\Omega\la\nabla\tilde\xi^\eps, D_\eps\nabla\tilde\xi^\eps \ra f^\eps\ud x\nonumber\\
& = &
\frac12\int_{\Omega}\la\nabla\tilde\xi,(\lbar{D}+\lbar{G})\nabla\tilde\xi \ra f\ud x
\nonumber\\
&&
+\frac12\int_{\Omega}\la\nabla\tilde\xi,(\lbar{D}+\lbar{G})\nabla\tilde\xi \ra (\bar{f_c}-f)\ud x+\lim_{\eps\to0}
\frac12\int_\Omega\la\nabla\tilde\xi^\eps, D_\eps\nabla\tilde\xi^\eps \ra (f_c-\bar{f}_c)\ud x\label{lusin}\\
&&+ \lim_{\eps\to0}\frac12\int_\Omega\la\nabla\tilde\xi^\eps, D_\eps\nabla\tilde\xi^\eps \ra (f-f_c)\ud x
+\lim_{\eps\to0}\frac12\int_\Omega\la\nabla\tilde\xi^\eps, D_\eps\nabla\tilde\xi^\eps \ra (f^\eps-f)\ud x.\label{egorov}
\end{eqnarray}

A final ingredient we need is that the sequence of functions $\la\nabla\tilde\xi^\eps, D_\eps\nabla\tilde\xi^\eps \ra$ is \emph{equi-integrable}:
\emph{for all $\sigma>0$, there exists a $\delta>0$ such that for any $S\subset\Omega$ with $|S| \leq \delta$, then}
\begin{equation}\label{equi.int}
\int_S \la\nabla\tilde\xi^\eps, D_\eps\nabla\tilde\xi^\eps \ra \leq \sigma
\,\,\,\text{\emph{holds for all $\eps> 0$.}}
\end{equation}
Once this is shown, we can then make use of Lusin and Egorov Theorems to claim that all the terms in \eqref{lusin} and \eqref{egorov} converge to zero as $\eps\to0$:
up to arbitrarily small measures, $f$ equals a continuous function $f_c$, and the convergence of $f^\eps$ to $f$ is uniform.
We recall again that $f^\eps$ and $f$ are uniformly bounded functions. 

We now show that the sequence of functions $\tilde{\xi}^\eps$ can be constructed so as it satisfies \eqref{equi.int}. Without loss of generality, we replace $\tilde{\xi}$ by a continuous and piece-wise affine function -- this can be achieved by an approximation using Galerkin or finite element method (given that $\tilde{\xi}$ is smooth). Then we have a partition of $\Omega$ into a collection of polyhedrons. For simplicity, we can further assume that these polyhedrons are the same
as the $C_j$ on each of which $\bar{f}_c$ is constant.
Now we construct $\tilde\xi^\eps$ according to the following procedure.

First, we define $A(x,y)=D(x,y)=\pi(x,y)B^{-1}(y)$. 
By the smooth assumption of $\pi$ and $B$, we have that $A$ is smooth in $y\in\mathbb T^n$ and $x\in C_j$.

Now, for $x\in C_j$, as $\nabla\tilde{\xi}$ is a constant vector $p_j\in\mathbb{R}^n$, the homogenized matrix $\lbar{A}(x)$ in Theorem \ref{GammaMainThm} is given by \eqref{eff.mat} and is repeated here for convenience.
\[
\big\la \lbar{A}(x) p_j, p_j\big\ra 
= \inf\left\{\int_{\mathbb T^n} 
\left\la A\left(x,y\right)(p_j + \nabla v),\, (p_j + \nabla v)\right\ra\ud y,
\quad v\in H^1(\mathbb T^n)\right\}.
\]
The $\inf$ above is achieved by $v_j(y)=|p_j|\hat{w}_j(x,y)$ where
$\hat{w}_j$ solves the following cell-problem:
\[
\text{div}_y\left(A(x,y)\nabla \hat{w}_j\right)=-\text{div}_y\left(A(x,y)\frac{p_j}{|p_j|}\right),\,\,\,\hat{w}_j(x,\cdot)\in H^1(\mathbb T^n),\,\,\,\int_{\mathbb T^n}\hat{w}_j(x,y)\,\ud y = 0.
\]
The smoothness assumption on $A$ implies that
\[
\|\hat{w}_j(x,\cdot),\,\,\,\nabla_y \hat{w}_j(x,\cdot),\,\,\,
\nabla_x \hat{w}_j(x,\cdot)\|_{L^\infty(\mathbb T^2)}
\leq C
\]
for some constant $C$ that does not depend on $x$ and $\eps$.

Next, let $0 < d_1 < d_2$ be two positive numbers. For each $C_j$, there exists a smooth subdomain $C_j'$ of $C_j$ such that
$d_1\eps \leq \text{dist}(\partial C_j', \partial C_j) \leq d_2\eps$.
Then we define a cut-off function $\eta^\eps_j$ on $C_j$ satisfying: 
(i) $0\leq \eta^\eps_j \leq 1$ on $C_j$;
(ii) $\eta^\eps_j = 1$ on $C_j'$; and 
(iii) $\eta^\eps_j(x)\longrightarrow0$ smoothly as $x\longrightarrow\partial C_j$ so that
$\eta^\eps_j \in C^\infty_0(C_j)$;
(iv) $\|\eps\nabla\eta^\eps_j\|_{L^\infty(C_j)} \leq C$ for an $\eps$-independent constant $C$.

With the above, suppose
$\tilde{\xi}(x) 
= \sum_j\big[\alpha_j + \la p_j,x\ra\big] 
\chi_{C_j}(x)$,
where $\chi_{C_j}$ is the characteristic function of $C_j$.
We then define
\[
\tilde{\xi}^\eps(x) = \sum_j\left[\alpha_j + \la p_j,x\ra + \eps\eta^\eps_j(x)|p_j|\hat{w}_j(x,\frac{x}{\eps})
\right]\chi_{C_j}(x)
\]
Then we have,
\[
\nabla\tilde{\xi}^\eps(x) =
\sum_j \left[p_j + 
\eta^\eps_j(x)|p_j|\nabla_y\hat{w}_j(x,\frac{x}{\eps})
+
\eps\eta^\eps_j(x)|p_j|\nabla_x\hat{w}_j(x,\frac{x}{\eps})
+
\eps\nabla\eta^\eps_j(x)|p_j|\hat{w}_j(x,\frac{x}{\eps})\right]\chi_{C_j}(x).
\]
By the aforementioned estimates for $\hat{w}_j$ and $\eta^\eps_j$, we can conclude that
$|\nabla \tilde{\xi}^\eps(x)| \leq C|p_j|$ for $x\in C_j$ and hence
\[
|\nabla \tilde{\xi}^\eps(x)| \leq C|\nabla\tilde{\xi}(x)|
\,\,\,\text{for all $x\in\Omega$.}
\]
(Here we make use of the $L^\infty(\mathbb T^n)$ estimates for $\hat{w}_j$ but we could also resort to the weaker $L^2(\mathbb T^n)$ estimates.) Note that the above statement holds uniformly for all $\eps\ll 1$. 
We can then conclude \eqref{equi.int} as
$\displaystyle \int_{\Omega}|\nabla\tilde{\xi}|^2 \ud x$ is finite.

The fact that $\left\{\tilde{\xi}^\eps\right\}_{\eps>0}$ is a recovery sequence for $\tilde{\xi}$ 
is due to the properties that 
$\tilde{\xi}^\eps\longrightarrow\tilde{\xi}$ in $L^2(\Omega)$ and 
$\nabla\tilde{\xi}^\eps$ differs from the ``optimal'' oscillatory functions
$\left\{p_j + |p_j|\nabla_y\hat{w}_j(x,\frac{x}{\eps})\right\}_j$ only on $\bigcup_j C_j\backslash C_j'$ which has vanishing measure as $\eps\longrightarrow0$. More precisely, we have
\begin{eqnarray*}
&&\lim_{\eps\to0}\int\big\langle A(x,\frac{x}{\eps})\nabla\tilde{\xi}^\eps,\nabla\tilde{\xi}^\eps\big\rangle\bar{f}_c\ud x
=
\lim_{\eps\to0}\sum_j\int_{C_j}\big\langle A(x,\frac{x}{\eps})\nabla\tilde{\xi}^\eps,\nabla\tilde{\xi}^\eps\big\rangle\bar{f}_{c_j}\ud x\\
&=&\sum_j\int_{C_j}\int_{\mathbb T^n}\Big\langle A(x,y)\big(p_j + |p_j|\nabla_y\hat{w}_j(x,y)\big),
\big(p_j + |p_j|\nabla_y\hat{w}_j(x,y)\big)\Big\rangle\ud y \,\bar{f}_{c_j}\ud x\\
&=&
\sum_j\int_{C_j}\langle \bar{A}(x)p_j,p_j\rangle\bar{f}_{c_j}\ud x
=
\int\big\langle \bar{A}(x)\nabla\tilde{\xi},\nabla\tilde{\xi}\big\rangle\bar{f}_c\ud x.
\end{eqnarray*}
The above computation is classical in the theory of two-scale convergence -- see \cite[Prop. 1.14(i), and equations (2.10), (2.11)]{allaire1992homogenization}. 
Note also that \eqref{GammaConstMult1} and \eqref{GammaConstMult2} hold
as $\bar{f}_c$ is constant on the $C_j$'s.

We can now conclude \eqref{approx.seq.psi}.

\bibliographystyle{alpha}
\bibliography{homo_bib}

\end{document}